\pdfoutput=1
\RequirePackage{ifpdf}
\ifpdf 
\documentclass[pdftex]{sigma}
\else
\documentclass{sigma}
\fi

\usepackage{tikz}
\usepackage{tikz-cd}
\usepackage{bm}
\usepackage{mathrsfs}
\usepackage{enumitem}
\usepackage{stmaryrd}

\numberwithin{equation}{section}

\newtheorem{thm}{Theorem}[section]
\newtheorem{cor}[thm]{Corollary}
\newtheorem{lem}[thm]{Lemma}
\newtheorem{prop}[thm]{Proposition}
 \theoremstyle{definition}
\newtheorem{dfn}[thm]{Definition}
\newtheorem{ex}[thm]{Example}
\newtheorem{rem}[thm]{Remark}
\newtheorem{exr}[thm]{Exercise}

\DeclareMathOperator{\Aut}{Aut}
\newcommand{\C}{\mathbb{C}}
\newcommand{\R}{\mathbb{R}}
\newcommand{\Z}{\mathbb{Z}}
\newcommand{\Oo}{\mathcal{O}}
\newcommand{\X}{\mathcal{X}}
\newcommand{\Xo}{\mathcal{X}^\circ}
\newcommand{\Mb}{\mathbf{M}}
\newcommand{\Mbo}{\mathbf{M}^\circ}
\newcommand{\cu}{\Sigma}
\newcommand{\Bb}{\mathbf{B}}
\newcommand{\Bbo}{\mathbf{B}^\circ}
\newcommand{\Xt}{\tilde{\mathcal{X}}}
\newcommand{\pib}{\bm{\pi}}
\newcommand{\Hig}{\mathbf{Higgs}}
\newcommand{\Ct}{\tilde{C}}
\newcommand{\bLambda}{\bm{\Lambda}}
\newcommand{\blambda}{\bm{\lambda}}
\newcommand{\gfr}{\mathfrak{g}}
\newcommand{\tfr}{\mathfrak{t}}
\newcommand{\Hit}{\bm{h}}
\newcommand{\VH}{\mathsf{V}}
\newcommand{\Autd}{\operatorname{Aut}_D}
\newcommand{\aD}{\tau}
\newcommand{\Cc}{\mathbf{C}}
\newcommand{\ADE}{\mathrm{ADE}}
\newcommand{\wt}{\mathsf{wt}}
\newcommand{\qf}{\mathbf{q}}
\newcommand{\tC}{\tilde{\bm{C}}}
\newcommand{\Higgs}{\bm{Higgs}}

\begin{document}

\allowdisplaybreaks

\newcommand{\arXivNumber}{1809.05736}

\renewcommand{\thefootnote}{}

\renewcommand{\PaperNumber}{001}

\FirstPageHeading

\ShortArticleName{Aspects of Calabi--Yau Integrable and Hitchin Systems}

\ArticleName{Aspects of Calabi--Yau Integrable\\ and Hitchin Systems\footnote{This paper is a~contribution to the Special Issue on Geometry and Physics of Hitchin Systems. The full collection is available at \href{https://www.emis.de/journals/SIGMA/hitchin-systems.html}{https://www.emis.de/journals/SIGMA/hitchin-systems.html}}}

\Author{Florian BECK}

\AuthorNameForHeading{F.~Beck}

\Address{FB Mathematik, Universit\"at Hamburg, Bundesstrasse 55, 20146 Hamburg, Germany}
\Email{\href{mailto:florian.beck@uni-hamburg.de}{florian.beck@uni-hamburg.de}}

\ArticleDates{Received September 25, 2018, in final form December 19, 2018; Published online January 01, 2019}

\Abstract{In the present notes we explain the relationship between Calabi--Yau integrable systems and Hitchin systems based on work by Diaconescu--Donagi--Pantev and the \mbox{author}.
Besides a review of these integrable systems, we highlight related topics, for example variations of Hodge structures, cameral curves and Slodowy slices, along the way.}


\Classification{14H70; 14D07; 14J32}

\renewcommand{\thefootnote}{\arabic{footnote}}
\setcounter{footnote}{0}

\section{Introduction}
From a geometric viewpoint, complex integrable systems are holomorphic symplectic mani\-folds~$(\Mb,\omega)$ that admit a holomorphic map $\pi\colon \Mb\to B$ to another complex manifold~$B$ such that the generic fibers of $\pi$ are compact connected Lagrangian submanifolds. These are isomorphic to complex tori, i.e., quotients $V/\Gamma$ of a complex vector space $V$ by a lattice $\Gamma\subset V$ of full rank, by the complex version of the Arnold--Liouville theorem \cite[Chapter~2]{DM2}. A general holomorphic symplectic manifold does not admit such a structure. Therefore complex integrable systems play a special role in holomorphic symplectic geometry.

The first purpose of this article is to introduce two intricate infinite families of complex integrable systems, namely Calabi--Yau integrable systems and Hitchin systems. The former were constructed by Donagi--Markman~\cite{DM1} from any complete family $\pi\colon \mathcal{X}\to B$ of compact Calabi--Yau threefolds. The generic fibers are Griffiths' intermediate Jacobians~$J^2(X_b)$~\cite{GriffithsI} of the smooth fibers~$X_b$ of~$\pi$. These complex tori are a generalization of the Jacobian of a~compact Riemann surface. The geometry of $J^2(X)$ for a compact Calabi--Yau threefold $X$ is in general poorly understood unlike in the case of Fano threefolds~\cite{ClemensGriffiths}. In mathematics, Calabi--Yau integrable systems have close links to Deligne cohomology and algebraic cycles (see \cite{delAngelMS,DM1}, \cite[Section 7.8]{EsnaultViehweg}). In mathematical physics, they play an important role in $F$-theory \cite{Tbranes,Tbranes2} and geometric transitions~\cite{DDD}.

Hitchin systems were first discovered by Hitchin \cite{Hit2,Hit1} and have been extensively studied since then. They are constructed from a compact connected Riemann surface~$C$ (of genus $\geq 2$) together with a semisimple complex Lie group $G$ via the moduli space $\Hig(C,G)$ of $G$-Higgs bundles on~$C$ (Section~\ref{s:hit}). The latter carries a very rich geometry, in particular it is a Hyperk\"ahler manifold.
The Hitchin system $\Hit\colon \Hig(C,G)\to \Bb(C,G)$ itself has very surprising features. For example, the generic fibers of $\Hit\colon \Hig(C,G)\to \Bb(C,G)$ are dual as abelian varieties to the generic fibers of $^L\Hit(C,G)\colon \Hig\big(C,^L G\big)\to \Bb\big(C,^L G\big)$ \cite{DP,HauselThaddeus}\footnote{This make sense because $\Bb(C,G)\cong \Bb(C, ^L G)$ canonically.}. Here $^L G$ is the Langlands dual group of~$G$ (e.g., if $G={\rm SL}(n,\C)$, then $^L G={\rm PGL}(n,\C)$). Proving this duality requires an explicit description of the generic fibers in terms of certain branched coverings, namely cameral curves (introduced in~\cite{Donagi2}), of~$C$.

Even though these two complex integrable systems are a priori of a very different nature, we relate the two following the pioneering papers \cite{DDD,DDP} and our own work \cite{Beck, Beck2}. In the latter we developed a~Hodge-theoretic framework in which both Calabi--Yau integrable systems and Hitchin systems fit naturally.

The second purpose of this article is to motivate this framework and to sketch how it is used to relate Calabi--Yau integrable systems with Hitchin systems in a precise way (Theorem \ref{thm:nccyhit}). For that reason, we mostly outline the strategy of the proofs or omit them altogether and refer to \cite{Beck,Beck2} for further details\footnote{An exception is Corollary \ref{cor:exactness} which appeared neither in~\cite{Beck} nor~\cite{Beck2}.}.

\subsection{Structure of the notes}
We begin by introducing complex integrable systems in Section~\ref{s:hodge}, and study them from a~Hodge-theoretic viewpoint. We define non-degenerate complex tori and give the intermediate Jacobian~$J^2(X)$ of a compact Calabi--Yau threefold $X$ as an example. The former are equivalent to certain Hodge structures of weight~$1$. Generalizing to the relative setting, we arrive at variations of Hodge structures of weight~$1$. These give rise to complex integrable systems if they admit a special section, namely an abstract Seiberg--Witten differential (Proposition~\ref{thm:abSW}). With these methods we construct Calabi--Yau integrable systems from complete families of compact Calabi--Yau threefolds in a different way than originally done by Donagi--Markman~\cite{DM1}.

The second prime examples of integrable systems, $G$-Hitchin systems, are introduced in Section~\ref{s:hit} for any general (semi)simple complex Lie group~$G$. This requires some background material, in particular the introduction of the adjoint quotient of a~semisimple complex Lie algebra and cameral curves. The latter are a necessary tool to determine the isogeny (and isomorphism) classes of generic fibers of $G$-Hitchin systems in general. When $G\subset {\rm GL}(n,\C)$ is a~classical semisimple complex Lie group, they are closely related to the spectral curves introduced by Hitchin~\cite{Hit2, Hit1} for the same purpose. We give a detailed comparison between the two notions based on~\cite{Donagi2}.

In the final Section~\ref{s:relation}, Calabi--Yau integrable systems are related to Hitchin systems. The corresponding families of (non-compact) Calabi--Yau threefolds are constructed via Slodowy slices. We review them together with their relation to so-called $\Delta$-singularities where $\Delta$ is an irreducible Dynkin diagram. Finally, we state the precise relationship between (non-compact) Calabi--Yau integrable systems and $G$-Hitchin systems where the simple complex Lie group $G$ has Dynkin diagram~$\Delta$. Moreover, we recover the above mentioned Langlands duality statement of Hitchin systems via Poincar\'e duality (more precisely, Poincar\'e--Verdier duality~\cite[Chapter~3]{KashiwaraSchapira} applied to the families of (non-compact) Calabi--Yau threefolds.

\section{Hodge theory of integrable systems}\label{s:hodge}
\looseness=-1 The generic fibers of an algebraic integrable system are torsors for abelian varieties (see Defi\-nition~\ref{dfn:AIS} and Lemma~\ref{lem:AIS} below). It follows that the smooth part of an algebraic integrable system is a family of abelian varieties if it admits a (Lagrangian) section. Such families are intimately related to Hodge theory, more precisely variations of Hodge structures. In the following we explain this relation and how holomorphic symplectic structures fit into the Hodge-theoretic framework.

\subsection{Complex integrable systems}
At first we generalize the notion of an algebraic integrable system by relaxing the condition on the generic fibers. This is necessary because Calabi--Yau integrable systems, one of our prime examples of integrable systems, are of this form. We begin by fixing the definition of an algebraic integrable system.

\begin{dfn}[algebraic integrable system]\label{dfn:AIS} Let $(\Mb,\omega)$ be a holomorphic symplectic manifold and $B$ a connected complex manifold. A holomorphic map $\pi\colon \Mb\to B $ is called an algebraic integrable system if
\begin{enumerate}[label=\roman*)]\itemsep=0pt
\item $\pi$ is proper and surjective,
\item its smooth fibers are Lagrangian and connected,
\item its smooth part $\pi^\circ\colon \Mb^\circ\to B^\circ$ admits a relative polarization, i.e., a line bundle $\mathcal{L}\to \Mb$ such that $\mathcal{L}_{|\Mb_b}$ is ample for all $b\in B^\circ$.
\end{enumerate}
We call an algebraic integrable system \emph{smooth} if $\pi\colon \Mb\to B$ is a submersion so that $B^\circ=B$.
\end{dfn}
The definition implies the following well-known properties of algebraic integrable systems (see~\cite{DM2}).
\begin{lem}\label{lem:AIS} Let $\pi\colon (\Mb,\omega)\to B$ be an algebraic integrable system. Then its smooth part $\pi^\circ\colon \Mbo\to B^\circ$ is a torsor for a family $\mathcal{A}(\pi^\circ)\to B$ of polarized abelian varieties. In particular, $(\Mb_b,\mathcal{L}_{|\Mb_b})$ is non-canonically isomorphic to a polarized abelian variety for every $b\in B^\circ$.
\end{lem}
\begin{proof}[Sketch of proof] Let $\mathcal{V}\to B^\circ$ be the vertical bundle of $\pi^\circ$, i.e., the bundle of vector fields on~$\Mb^\circ$ which are tangent to the fibers. Let $v,w\colon U\to \mathcal{V}$ be two local sections. If $U$ is small enough, then $v$, $w$ are Hamiltonian vector fields $v=X_{\pi^*f}$, $w=X_{\pi^*g}$ for functions $f,g\colon U\to \C$. Since the fibers of $\pi^\circ$ are Lagrangian, it follows that $[v,w]=X_{\omega(v,w)}=0$. Hence the flows of sections of $\mathcal{V}$ define a fiber-preserving action of the abelian group $(\mathcal{V},+)$ associated to the commutative Lie algebra $(\mathcal{V}, [\bullet,\bullet])$ on $\Mb^\circ$. The submanifold
\begin{gather*}
\Gamma=\big\{ v\in \mathcal{V}\,|\,\exists\, x\in \Mb^\circ \colon v\cdot x=x \big\}\subset \mathcal{V}
\end{gather*}
intersects each fiber in a full lattice. The relative polarization of $\pi^\circ\colon \Mb^\circ\to B^\circ$ is transported to $\mathcal{A}(\pi^\circ):=\mathcal{V}/\Gamma\to B^\circ$ which is therefore a family of abelian varieties acting simply transitively on $\pi^\circ\colon \Mb^\circ\to B^\circ$.
\end{proof}

\begin{rem}The holomorphic symplectic form $\omega$ induces an isomorphism $T^*B^\circ \to \mathcal{V}$ under which $\Gamma\subset \mathcal{V}$ becomes a Lagrangian submanifold of $T^*B^\circ$ with its canonical symplectic structure. At least locally, $T^*B^\circ/\Gamma\to B^\circ$ and $\pi^\circ\colon \Mb^\circ \to B^\circ$ are therefore isomorphic to each other.
\end{rem}

We now generalize algebraic integrable systems by allowing more general complex tori as generic fibers. Let $T=V/\Gamma$ be a complex torus where $\Gamma\subset V$ is a~full lattice in a vector space~$V$ of dimension $\dim_{\C}V=g$. We denote by $NS(T)$ the N\'eron--Severi group of~$T$. Then there is a~canonical isomorphism~\cite{BLAbelian}
\begin{align*}
NS(V/\Gamma)&\cong \big\{ H\colon V\otimes_{\C} \bar{V}\to \C~\mbox{Hermitian form}\,|\, \operatorname{im}(H)(\Gamma,\Gamma)\subset \Z \big\}, \\
c_1(L)&\mapsto H_L
\end{align*}
for $L\in \operatorname{Pic}(T)$. The Hermitian form $H_L$ satisfies the following two Riemann bilinear relations iff $L$ is ample:
\begin{enumerate}[label=(\Roman*)]\itemsep=0pt
\item \label{eq:RiemannI} $\operatorname{im}(H_L)\colon V_{\R}\otimes_{\R}V_{\R} \to \R$ is non-degenerate and satisfies $\operatorname{im}(H_L)(\Gamma,\Gamma)\subset \Z$,
\item \label{eq:RiemannII} $H_L(v,v)> 0$ for all $v\neq 0$.
\end{enumerate}

\newpage

\looseness=-1 To relax the second Riemann bilinear relation, we recall the index of a non-degenerate Hermitian form $H$ on a complex vector space $V$: Let $V^-\subset V$ be the maximal vector subspace such that the restriction of $H$ to $V^-$ is negative definite. Then the index of~$H$ is the dimension of~$V^-$.

\begin{dfn}\label{dfn:poltorus} A complex torus $T=V/\Gamma$ is a non-degenerate complex torus of index $k\geq 0$ if it admits a line bundle $L\to T$ such that the Hermitian form $H_L$ satisfies the first Riemann bilinear relation~\ref{eq:RiemannI} and is of index\footnote{The index of a non-degenerate Hermitian form $H$ on the complex vector space $V$ is the dimens.}~$k$.
\end{dfn}
\begin{rem}If $L\to T$ is a holomorphic line bundle of index $k$, then
\begin{gather*}
H^i(T,L)\begin{cases}
\neq 0, & i=k, \\
=0, & i\neq k,
\end{cases}
\end{gather*}
see \cite{BLTori}. If $k=0$, then $H^0(T,L)$ is generated by the corresponding theta functions.
\end{rem}

\begin{ex}Every complex torus $T$ with Picard number $\rho(T)=\operatorname{rk}(NS(T))=0$ is degenerate. Such tori exist in dimension $\geq 2$.
\end{ex}

\begin{ex}\label{ex:nondeg} Let $X$ be a compact K\"ahler manifold of dimension $\dim_{\C}X=m=2n-1$ and $H^{m}(X,\C)=\oplus_{p} H^{p,m-p}(X)$ the Hodge decomposition. Then Griffiths' intermediate Jacobian~\cite{GriffithsI} of $X$ is defined by
\begin{gather*}
J^n(X):=H^{2n-1}(X,\C)/\big( F^nH^{2n-1}(X,\C)+H^{2n-1}(X,\Z)\big),
\end{gather*}
where $F^nH^{m}(X,\C)=\oplus_{p\geq n} H^{p,m-p}(X)$, also see~(\ref{eq:Hdgfiltration}). It is a non-degenerate complex torus of index $k>0$ if the canonical bundle $K_X$ of $X$ is trivial, cf.\ Example~\ref{ex:j2}.
\end{ex}

\begin{dfn}[complex integrable systems of index $k$] A holomorphic map $\pi\colon (\Mb,\omega)\to B$ between a holomorphic symplectic manifold $(\Mb, \omega)$ and a connected complex manifold $B$ is called a complex integrable system of index~$k$, $k\geq 0$, if
\begin{enumerate}[label=\roman*)] \itemsep=0pt
\item $\pi$ is proper and surjective,
\item its smooth fibers are Lagrangian and connected,
\item its smooth part $\pi^\circ\colon \Mbo\to B^\circ$ admits a relative polarization of index $k\geq 0$.
\end{enumerate}
We call $\pi\colon \Mb\to B$ a smooth complex integrable system if $B^\circ=B$.
\end{dfn}
Clearly, an algebraic integrable system is a complex integrable system of index~$0$. We show in Section~\ref{s:cyis} that intermediate Jacobians~$J^2(X)$ of compact Calabi--Yau threefolds~$X$ define complex integrable systems of index $k>0$.

\subsection{VHS and smooth complex integrable systems}\label{ss:VHSsmooth}
From now on, we concentrate on \emph{smooth} complex integrable systems $\pi\colon (\Mb,\omega)\to B$ which admit a (Lagrangian) section $s\colon B\to \Mb$. In particular, $\pi\colon \Mb\to B$ is a family of complex tori.

\begin{rem}The requirement that the smooth complex integrable system $\pi\colon \Mb\to B$ admits a~section is not very restrictive because we can always work with the associated family $\mathcal{A}(\pi)\to B$ of polarized complex tori, cf.\ Lemma~\ref{lem:AIS}.
\end{rem}

To see how Hodge theory is related to smooth complex integrable systems, we begin with a~single complex torus $T\cong \C^g/\Gamma$ which is in particular a compact K\"ahler manifold with K\"ahler form $\omega_T=\sum\limits_{i=1}^g {\rm d}z_i \wedge {\rm d}\bar{z}_i$. It follows that for each $n=1,\dots, 2g$, the cohomology group~$H^n(T,\C)$ admits the Hodge decomposition
\begin{gather}\label{eq:Hk}
H^n(T,\C)=\bigoplus_{p+q=n} H^{p,q}(T),\qquad H^{q,p}(T)=\overline{H^{p,q}(T)}.
\end{gather}
The pair $H^n(T,\C)$ together with the decomposition~(\ref{eq:Hk}) of $H^n(T,\C)$ is an example of an integral Hodge structure ($\Z$-Hodge structure) of weight~$n$. The K\"unneth formula implies that
\begin{gather*}
H^n(T,\Z)\cong H^1(T,\Z)^{\otimes n},
\end{gather*}
so that the only interesting cohomology group is $H^1(T,\Z)$ with Hodge decomposition
\begin{gather*}
H^1(T,\C)=H^{1,0}(T)\oplus H^{0,1}(T).
\end{gather*}
It uniquely determines the complex torus $T$.
\begin{lem}\label{lem:JT}Let $T$ be a complex torus and $H_{\Z}=H^1(T,\Z)$ the corresponding $\Z$-Hodge structure of weight~$1$. Then
\begin{gather*}
T\cong J(H_{\Z})^\vee, \qquad J(H_{\Z}):= H^1(T,\C)/\big(H^{1,0}(T)+H^1(T,\Z)\big),
\end{gather*}
where the superscript $\vee$ stands for the dual torus.
\end{lem}
\begin{proof}See Exercise \ref{exr:Albanese}.
\end{proof}

The subspace $H^{1,0}(T)\subset H^1(T,\C)$ is the simplest non-trivial example of a Hodge filtration which for a general $\Z$-Hodge structure $H_{\Z}$ of weight $n$ is defined by
\begin{gather}\label{eq:Hdgfiltration}
F^pH_{\C}=\bigoplus_{l\geq p } H^{l,n-l},\qquad 0\leq p \leq n,
\end{gather}
for $H_{\C}=H_{\Z}\otimes \C$. It is equivalent to the Hodge decomposition via $H^{pq}=F^p\cap \bar{F}^q$ and the condition $F^p\cap \bar{F}^{n-p+1}=0$. We refer to the pair $(H_{\Z},F^\bullet H_{\C})$ as an integral Hodge structure of weight $n$ as well.

If $T$ is an abelian variety, then the $\Z$-Hodge structure $H^1(T,\Z)\subset H^1(T,\C)$ admits an additional structure, namely the non-degenerate bilinear form
\begin{gather*}
Q\colon \ H^1(T,\Z)\otimes H^1(T,\Z) \to \Z, \qquad
Q(\alpha, \beta)=\int_{T} \alpha\wedge \beta \wedge \omega_T^{\wedge g-1}.
\end{gather*}
It defines the positive definite Hermitian form\footnote{The factor $2$ is uncommon but convenient for our purposes, see the proof of Lemma~\ref{lem:HQ}.} $H_Q(\alpha,\beta):=2{\rm i}Q(\alpha, \bar{\beta})$ on $H_{\C}$. If $T$ is not algebraic, so there is no K\"ahler class in $H^2(T,\Z)$, then $Q$ is not defined over~$\Z$. A weaker version is the following which is the Hodge-theoretic analogue of non-degenerate complex tori of index $k\geq 0$.

\begin{dfn}\label{dfn:Q}Let $(H_{\Z},H_{\C})$ be an integral Hodge structure of weight~$1$. A skew-symmetric bilinear form $Q\colon H_{\Z}\otimes H_{\Z}\to \Z$ is a polarization of index $k\geq 0$ if
\begin{enumerate}[label=\roman*)]\itemsep=0pt
\item \label{dfn:Qi} $Q(\alpha,\beta)=-Q(\beta,\alpha)$,
\item \label{dfn:Qii} $Q\big(H^{1,0},H^{1,0}\big)=0=Q\big(H^{0,1},H^{0,1}\big)$,
\item \label{dfn:Qiii} $H_Q(\alpha,\beta):=2{\rm i}Q(\alpha, \bar{\beta})$ is a non-degenerate Hermitian form of index~$k$ on~$H_{\C}$.
\end{enumerate}
The pair $(H_\Z,Q)$ is called a polarized $\Z$-Hodge structure of weight $1$ and index~$k$.
\end{dfn}

\begin{rem} The notion of a polarized $\Z$-Hodge structure is traditionally reserved for the case of index $0$ (which in the case of weight $1$ corresponds to abelian varieties). However, we slightly weaken this notion and speak of polarized/polarizable $\Z$-Hodge structures of weight $1$ even if the index is positive.
\end{rem}
The next lemma gives the relation between Definitions~\ref{dfn:poltorus} and~\ref{dfn:Q}.
\begin{lem}\label{lem:HQ} Let $H_{\Z}$ be a $\Z$-Hodge structure of weight $1$ and $Q\colon H_{\Z}\otimes H_{\Z}\to \Z$ a polarization of index $k\geq 0$. Then the associated Hermitian form $H_Q$ $($see Definition~{\rm \ref{dfn:Q}\ref{dfn:Qiii})} defines a~polarization of index $k$ on the complex torus $J(H_{\Z})=H_{\C}/\big(F^1H_{\C}+H_{\Z}\big)$.
\end{lem}
\begin{proof} Define $V=H_{\C}/F^1H_{\C}\cong\overline{F^1H_{\C}}$ and $\Gamma=H_{\Z}$. Observe that $\Gamma\hookrightarrow V$ via
\begin{gather*}
\gamma=\overline{\gamma^{0,1}}+\gamma^{0,1}\mapsto \gamma^{0,1}, \qquad\gamma^{0,1}\in H^{0,1},
\end{gather*}
which induces the isomorphism $J(H_{\Z})\cong V/\Gamma$. It remains to show that $\operatorname{im}(H_Q)$ is $\Z$-valued on $\Gamma\hookrightarrow V$. Let $\gamma=\overline{\gamma^{0,1}}+\gamma^{0,1}$, $\delta=\overline{\delta^{0,1}}+\delta^{0,1}\in H_{\Z}\subset H_{\C}$ so that
\begin{gather*}
Q(\gamma,\delta)=2 \operatorname{re}\big( Q\big(\gamma^{0,1},\overline{\delta^{0,1}}\big)\big)\in \Z.
\end{gather*}
Hence under the inclusion $\Gamma\subset V$, $\gamma\mapsto \gamma^{0,1}$, we have
\begin{gather*}
\operatorname{im}(H_Q)(\gamma,\delta)=2 \operatorname{im}\big( {\rm i} Q\big(\gamma^{0,1},\overline{\delta^{0,1}}\big)\big)=Q(\gamma,\delta)\in \Z
\end{gather*}
and the claim is proven.
\end{proof}

Therefore polarized $\Z$-Hodge structure of weight $1$ and index $k$ are equivalent to non-de\-ge\-ne\-ra\-te complex tori of index~$k$, cf.\ Lemma~\ref{lem:JT}.

\begin{ex}\label{ex:j2} Let $X$ be a compact K\"ahler manifold of dimension $\dim_{\C} X=3$ and $J=J^2(X)$ its (Griffiths') intermediate Jacobian, as defined in Example~\ref{ex:nondeg}. As a complex torus it is given by $J=V/\Gamma$ for $\Gamma= H^3(X,\Z)\subset V=\overline{F^2H^3(X,\C)}$. The $\Z$-Hodge structure $H^1(J,\Z)$ of weight~$1$ is determined by
\begin{gather}\label{eq:F1HJ}
H^{1}(J,\C)=F^1H^1(J,\C)\oplus \overline{F^1H^1(J,\C)}=F^2H^{3}(X,\C)\oplus \overline{F^2H^{3}(X,\C)}.
\end{gather}
It carries the polarization
\begin{gather*}
Q(\alpha,\beta)=\int_X \alpha\wedge \beta
\end{gather*}
which is of index $h^{0,1}(X)+h^{0,3}(X)$. In particular, $J^2(X)$ is a non-degenerate complex torus of index $\geq 1$ if $K_X\cong \Oo_X$, see \cite[Chapter~4]{BLTori}.
\end{ex}
The previous discussion generalizes to the family setting, i.e., to smooth complex integrable systems $\pi\colon \Mb\to B$ of index $k$ that admit a~section. Then the integer cohomology groups $H^1(M_b,\Z)$, $b\in B$, form a~locally constant sheaf $\VH_{\Z}(\pi)$ over $B$ and the fiberwise polarizations~$Q_b$ determine a morphism $Q\colon \VH_{\Z}\otimes \VH_{\Z}\to \underline{\Z}_B$ for the constant sheaf $\underline{\Z}_B$ on~$B$.
The induced holomorphic bundle $\VH_{\Oo}(\pi):=\VH_{\Z}(\pi)\otimes \Oo_B$ carries a~canonical flat holomorphic connection~$\nabla$, the Gau\ss-Manin connection.

Griffiths has proven~\cite{GriffithsI} that the Hodge filtrations $F^1H^1(M_b,\C)\subset H^1(M_b,\C)$ form a~holomorphic subbundle $F^1\VH_{\Oo}(\pi)\subset \VH_{\Oo}(\pi)$. The datum
\begin{gather*}
\VH(\pi)=(\VH_{\Z}(\pi),F^\bullet\VH_{\Oo}(\pi))
\end{gather*}
is an example of an integral variation of Hodge structure~\cite{GriffithsI} of weight $1$ which admits a~polarization $Q$ of index~$k$.

Conversely, let $(\VH,Q)$ be a polarized $\Z$-VHS of weight $1$ and index $k$ over $B$. Then
\begin{gather}
\mathcal{J}(\VH):=\VH_\Oo/\big(F^1\VH_{\Oo}+\VH_{\Z}\big)\to B
\end{gather}
is a family of polarized complex tori of index $k$. In that way, we see that families $\Mb\to B$ of complex tori (of index $k$) are equivalent to polarized $\Z$-VHS $(\VH,Q)$ of weight $1$ (and index~$k$). The vertical bundle $\mathcal{V}$ of $\mathcal{J}(\VH)\to B$ is canonically isomorphic to $\VH_{\Oo}/F^1\VH_{\Oo}$. The polarization~$Q$ induces the isomorphism $\psi_Q\colon \mathcal{V}\to F^1\VH_{\Oo}^*$.

\subsection{Abstract Seiberg--Witten differentials}\label{ss:absw}
It is a natural question if there are sufficient conditions on the polarized $\Z$-VHS $(\VH,Q)$ of index~$k$ such that $\mathcal{J}(\VH)\to B$ is a smooth complex integrable system of index~$k$. One answer is given by the next theorem.

\begin{prop}[\cite{Beck}]\label{thm:abSW} Let $B$ be a complex manifold and $(\VH,Q)$ be a polarized $\Z$-VHS of weight~$1$ and index~$k$ over~$B$. Assume there exists a global section $\bm{\lambda}\in H^0(B,\VH_{\Oo})$ such that
\begin{gather}\label{eq:condabsw}
\phi_{\bm{\lambda}}\colon \ TB\to F^1\VH_{\Oo}, \quad v\mapsto \nabla_v \bm{\lambda}
\end{gather}
is an isomorphism and denote $\iota=\phi_{\bm{\lambda}}^*\circ \psi_Q\colon \mathcal{V}\to T^*B$. Then there exists a unique symplectic form $\omega_{\bm{\lambda}}$ on $\mathcal{J}(\VH)\to B$ such that the zero section becomes Lagrangian and which induces~$\iota$. It is independent of the polarization $Q$ up to symplectomorphisms. Moreover, the same result holds true if $\VH$ is replaced by any other $\Z$-VHS $\VH'$ of weight $1$ which is isogenous to $\VH$.
\end{prop}
Here two $\Z$-VHS $\VH,\VH'$ of weight $1$ over $B$ are called isogenous if there exists a morphism $\psi\colon \VH\to \VH'$ such that the induced morphism $\mathcal{J}(\VH)\to \mathcal{J}(\VH')$ is a fiberwise isogeny.
\begin{proof}[Sketch of proof] The basic idea is to use $\blambda$ to prove that $\Gamma:=\iota(\VH_{\Z})\subset T^*B$ is Lagrangian so that the canonical symplectic structure on $T^*B$ descends to $T^*B/\Gamma\cong \mathcal{J}(\VH)$, also see the proof of Corollary~\ref{cor:exactness}.
\end{proof}

\begin{dfn}\label{dfn:abstractSW} Let $(\VH,Q)$ be a polarized $\Z$-VHS of weight $1$ over the complex manifold~$B$. A section $\bm{\lambda}\in H^0(B,\VH_{\Oo})$ satisfying the condition~(\ref{eq:condabsw}) is called an abstract Seiberg--Witten differential.
\end{dfn}

\begin{rem}The terminology is motivated by Seiberg--Witten theory in mathematical physics (see~\cite{DonagiSW} for an introduction). A~key ingredient in this theory are families $\mathcal{E}_{\rm SW}\to B$ of Seiberg--Witten curves. Mathematically, these are families of (generically smooth) plane elliptic curves. As a concrete example, we take
\begin{gather*}
\mathcal{E}_{\rm SW}=\mathcal{E}\colon \ y^2z=(x-1)(x+1)(x-u),\qquad u\in \C.
\end{gather*}
Let $U=\C\setminus \{1,-1\}\subset \C$ be the locus of smooth fibers and $\VH$ the corresponding $\Z$-VHS of weight~$1$. For each $u\in U$, $\omega_u=\tfrac{{\rm d}x}{y}_{|E_u}$ defines a basis of holomorphic $1$-forms, hence a~frame of~$F^1\VH_{\Oo}$. The Seiberg--Witten differential is defined by $\tfrac{{\rm d}x}{y}(x-u)_{|E_u}$. It is a meromorphic $1$-form with a single pole at~$\infty$. However, its residue vanishes so that
\begin{gather*}
\lambda_{\rm SW}(u)=\left[\frac{{\rm d}x}{y}(x-u)_{|E_u}\right]\in H^1(E_u, \C)
\end{gather*}
is well-defined. The section $\lambda_{\rm SW}\in H^0(U,\VH_{\Oo})$ has the property of an abstract Seiberg--Witten differential, namely
\begin{gather*}
T_uU \ni \partial_u \mapsto \nabla_{\partial_u} \lambda_{\rm SW}=-\tfrac{1}{2}\omega_u \in F^1\VH_{\Oo,u}
\end{gather*}
defines an isomorphism $TU\cong F^1\VH_{\Oo}$. Abstracting this property motivated Definition~\ref{dfn:abstractSW}.
\end{rem}

As the following Corllary shows, the cohomology class of the holomorphic symplectic form is an obstruction to the existence of an abstract Seiberg--Witten differential.
\begin{cor}\label{cor:exactness} Let $(\VH, Q)$ be a polarized $\Z$-VHS of weight $1$ over the complex manifold $B$ which admits an abstract Seiberg--Witten differential $\lambda\in H^0(B, \VH_{\Oo})$. Then the induced holomorphic symplectic structure on the total space of $\mathcal{J}(\VH)$ is exact.
\end{cor}

\begin{proof} Under the isomorphism $\iota\colon \mathcal{V}\to T^*B$, $\omega_{\blambda}$ (pulled back to $\mathcal{V}$) corresponds to the canonical symplectic structure $d\eta$ on $T^*B$ for the tautological $1$-form~$\eta$. The latter corresponds to
\begin{gather*}
T_{[s]} \mathcal{V} \ni w\mapsto Q(s, \nabla_{{\rm d}p(w)} \lambda), \qquad [s]\in \mathcal{V}\cong \VH_{\Oo}/F^1\VH_{\Oo},
\end{gather*}
where $p\colon \mathcal{V}\to B$ is the projection. However, this $1$-form does not descend to $\mathcal{J}(\VH)$. Instead we define $f\colon \mathcal{V}\to \C,$ $f([s])=Q(s,\lambda)$, and
\begin{gather*}
\gamma(w):=Q(s,\nabla_{{\rm d}p(w)}\lambda) - {\rm d}f_{[s]}(w),\qquad w\in T_{[s]}\mathcal{V}.
\end{gather*}
It follows that $\gamma(w)=0$ for $w\in T\VH_{\Z}\hookrightarrow T\mathcal{V}$ so that $\gamma$ descends to a $1$-form on $\mathcal{J}(\VH)$ with ${\rm d}\gamma=\omega_{\blambda}$.
\end{proof}

\subsection{Compact Calabi--Yau integrable systems}\label{s:cyis}
As an application of Theorem \ref{thm:abSW} we reprove a result by Donagi--Markman~\cite{DM1} and show how intermediate Jacobians of compact Calabi--Yau threefolds give rise to complex integrable systems. Since the notion of a compact Calabi--Yau threefold can be ambiguous, we first fix the following.

\begin{dfn}\label{dfn:cCY} A compact Calabi--Yau threefold is a compact K\"ahler manifold~$X$ of \linebreak \mbox{$\dim_{\C}X=3$} with trivial canonical bundle $K_X\cong \Oo_X$ and $H^1(X,\C)=0$.
\end{dfn}

Any family $\pi\colon \mathcal{X}\to B$ of compact Calabi--Yau threefolds determines a polarized $\Z$-VHS $\VH(\pi)$ of weight $1$ of index $1$ which is fiberwise given by (\ref{eq:F1HJ}) in Example~\ref{ex:j2}. Then the intermediate Jacobian fibration is defined by
\begin{gather}\label{eq:calJ2}
\mathcal{J}^2(\pi):=\mathcal{J}(\VH(\pi))\to B
\end{gather}
also denoted $\mathcal{J}^2(\mathcal{X})$. It is a family of non-degenerate complex tori of index~$1$. A necessary condition for~(\ref{eq:calJ2}) to carry a Lagrangian structure is
\begin{gather}\label{eq:condB}
\dim_{\C} B=h^{1,2}(X_b)+1,\qquad \forall\, b\in B.
\end{gather}
Following Donagi--Markman, this is achieved as follows: Let $\mathcal{X}\to B$ be a complete family of compact Calabi--Yau threefolds, i.e., the Kodaira--Spencer map
\begin{gather*}
\kappa_b\colon \ T_bB\to H^1(X_b,TX_b)
\end{gather*}
is an isomorphism for all $b\in B$. Then the dimension of $B$ is $h^{1,2}(X_b)$ for $b\in B$. To satisfy condition (\ref{eq:condB}), we consider the $\C^*$-bundle
\begin{gather*}
\rho\colon \ \tilde{B}\to B,\qquad \rho^{-1}(b)=H^0(X_b,K_{X_b})\setminus \{0\}
\end{gather*}
of non-zero holomorphic volume forms and denote its points by $(X_b,s_b)$, $s_b\in H^0(X_b,K_{X_b})\setminus \{0\}$. Not only does the pullback family $\Xt:=\rho^*\X\to \tilde{B}$ satisfy the dimension condition~(\ref{eq:condB}) but it even induces a complex integrable system:

\begin{thm}[\cite{DM1}]\label{thm:DM} Let $\pi\colon \X \to B$ be a complete family of compact Calabi--Yau threefolds and $\rho\colon \tilde{B}\to B$ be the $\C^*$-bundle of holomorphic volume forms as well as $\tilde{\pi}\colon \tilde{\X}=\rho^*\X\to \tilde{B}$ the pullback of~$\mathcal{X}$. Then $\mathcal{J}^2\big(\Xt\big)\to \tilde{B}$ carries the structure of a complex integrable system of index~$1$.
\end{thm}
\begin{rem}
The existence of complete families is a non-trivial fact and follows from the Bogomolov--Tian--Todorov theorem \cite{Bogomolov, Tian,Todorov}, i.e., the unobstructedness of compact Calabi--Yau threefolds.
\end{rem}
\begin{proof}[Sketch of proof]We outline a proof following \cite[Section~3.2]{Beck-thesis}; for the original proof see~\cite{DM1}. Let $\VH=\VH(\pi)$ and $\tilde{\VH}=\VH(\tilde{\pi})$ be the polarized $\Z$-VHS of weight $1$ and index $1$ determined by~$\pi$ and~$\tilde{\pi}$ respectively. The latter admits the tautological section $\bm{s}\colon \tilde{B}\to F^1\tilde{\VH}_{\Oo}$ defined by
\begin{gather*}
\bm{s}\colon \ \tilde{B}\to F^1\tilde{\VH}_{\Oo},\qquad (X,s)\mapsto s.
\end{gather*}
The completeness of the family $\pi\colon \X\to B$ implies that the morphism
\begin{gather*}
T\tilde{B}\to F^1\tilde{\VH}_{\Oo},\qquad v\mapsto \tilde{\nabla}_v \bm{s},
\end{gather*}
is an isomorphism, i.e., $\bm{s}$ is an abstract Seiberg--Witten differential for~$\tilde{\VH}$. Therefore $\tilde{J}^2(\X)\to \tilde{B}$ carries the structure of a complex integrable system by Theorem~\ref{thm:abSW}.
\end{proof}

\begin{ex}The simplest example is given by rigid Calabi--Yau threefolds $X$, i.e., $H^1(X,T_X)$ $=0$ for the holomorphic tangent bundle $T_X$ of $X$. It follows that $H^{2,1}(X)=0$ so that $J^2(X)=H^3(X,\C)/\big(H^{3,0}(X)+H^3(X,\Z)\big)$ is an elliptic curve. These are the only examples when~$J^2(X)$ is an abelian variety. The Calabi--Yau integrable system becomes trivial: $\Xt=X\times \C^*$ and
\begin{gather*}
\mathcal{J}^2(\Xt)=J^2(X)\times \C^*\to \C^*.
\end{gather*}
The holomorphic symplectic form is induced by ${\rm d}w \wedge \tfrac{{\rm d}z}{z}$ on $\C\times \C^*$ when we choose an isomorphism $H^3(X,\C)\cong \C$.
\end{ex}

\subsection{Exercises}
\begin{exr}Let $(M,\omega)$ be a projective K3 surface, i.e., $M$~is a connected simply-connected compact K\"ahler surface and $K_M\cong \mathcal{O}_M$ via $\omega$. Let $B$ be a compact Riemann surface and $\pi\colon M\to B$ a~surjective holomorphic map. Show that $\pi\colon M\to B$ is an algebraic integrable system and $B\cong \mathbb{CP}^1$.
\end{exr}

\begin{exr}Let $T=V/\Gamma$ be a complex torus for $V=\C^g$. Rephrase the Riemann bilinear relations~\ref{eq:RiemannI},~\ref{eq:RiemannII} in terms of the lattice $\Gamma$.
\end{exr}

\begin{exr}\label{exr:Albanese} Prove Lemma~\ref{lem:JT}.
\end{exr}

\begin{exr}Show that any compact Calabi--Yau threefold (Definition \ref{dfn:cCY}) is projective.
\end{exr}

\begin{exr}If $T=V/\Gamma$ is a complex torus, then its dual complex torus is defined by $T^\vee:=\bar{V}^*/\Gamma^\vee$ where $\bar{V}^*$ are the $\C$-antilinear forms on $V$ and $\Gamma^\vee=\{ \alpha\in \bar{V}^*\,|\,\alpha(\gamma)\in \Z$ $\forall\, \gamma\in \Gamma\}$. Show that $J^2(X)\cong J^2(X)^\vee$ for a compact Calabi--Yau threefold~$X$.
\end{exr}

\section[$G$-Hitchin systems and cameral curves]{$\boldsymbol{G}$-Hitchin systems and cameral curves}\label{s:hit}
Hitchin systems are very rich and intricate examples of \emph{algebraic} integrable systems, i.e., complex integrable systems of index~$0$. They are associated to any pair $(C,G)$ consisting of a compact Riemann surface $C$ of genus $\geq 2$ and any semisimple complex Lie group $G$ (see \cite{Donagi,DG,Faltings, Hit2,Hit1} as well as \cite[Chapter~4]{Beck-thesis} and~\cite{DalakovLectures} for an overview and further references).
The total space of the integrable is the smooth locus $\Hig(C,G)$ of the moduli space of semistable $G$-Higgs bundles over $C$ which are topologically trivial.
A (topologically trivial) $G$-Higgs bundle is a pair $(P,\varphi)$ composed of a (topologically trivial) $G$-bundle $P\to C$ and a section $\varphi\in H^0(C,K_C\otimes \operatorname{ad}(P))$, called Higgs field.

In this section we give a brief introduction to $G$-Hitchin systems, focusing on general semisimple complex Lie groups~$G$. This requires some preparations in Lie theory, most prominently the adjoint quotient of a semisimple complex Lie algebra. After that we give the isogeny class\footnote{It is possible to determine the \emph{isomorphism} class of $G$-Hitchin systems but this is beyond the scope of these notes, see Remark~\ref{rem:isogenyclass}.} of a generic fiber of $G$-Hitchin systems, called Hitchin fibers, in terms of cameral curves \cite{Donagi2, Donagi}. These are branched Galois coverings of $C$.

If $G\subset {\rm GL}(n,\C)$ is a classical semisimple complex Lie group, then cameral curves parametrize the eigenvalues of Higgs fields together with all possible orderings. This is in contrast to spectral curves \cite{Hit2,Hit1} which parametrize eigenvalues with a fixed ordering. Spectral curves are very intuitive and convenient to work with, especially for explicit computations, and are sufficient to determine the isomorphism classes of generic Hitchin fibers if~$G$ is a classical semisimple complex Lie group.

However, there are some conceptual issues with spectral curves. For example, if~$G$ is a general semisimple complex Lie group, then the definition of a spectral curve depends on a choice of representation of~$G$. This and further issues, even for classical $G$, are remedied by cameral curves as we explain in detail in Section~\ref{ss:cameral} based on~\cite{Donagi2}.

\subsection{Adjoint quotient}
Any semisimple complex Lie group $G$ acts on its Lie algebra $\gfr=\operatorname{Lie}(G)$ by the adjoint representation $\operatorname{Ad}\colon G\to {\rm GL}(\gfr)$. In particular, $G$ acts on the algebra $\C[\gfr]$ of polynomial functions on~$\gfr$ and the inclusion $\C[\gfr]^G\hookrightarrow \C[\gfr]$ defines the adjoint quotient
\begin{gather*}
\chi\colon \ \gfr \to \gfr\sslash G:=\operatorname{Spec}\big(\C[\gfr]^G\big).
\end{gather*}
Chevalley (see~\cite[Chapter~23]{Humphreys}) has proven that the pullback under the inclusion $\tfr\hookrightarrow \gfr$ of a Cartan subalgebra restricts to the isomorphism $\C[\gfr]^G\cong \C[\tfr]^W$ for the corresponding Weyl group~$W$. This implies that $\gfr\sslash G \cong \tfr/W$. Since the invariant polynomials of a~reflection group (e.g., the Weyl group~$W$) form a free polynomial algebra, there exist $\chi_1,\dots, \chi_r\in \C[\gfr]^G$ such that
\begin{gather*}
\C[\gfr]\cong \C[\chi_1,\dots, \chi_r]
\end{gather*}
and hence $\gfr\sslash G \cong \C^r$ non-canonically. Even though the free generators are not unique, their degrees $d_j:=\operatorname{deg}(\chi_j)$ are. The numbers $d_j-1$ are called the exponents of~$\gfr$.

The adjoint quotient $\chi\colon \gfr\to \tfr/W$ can be expressed more concretely. Let $v=v_{n}+v_s$ be the Jordan decomposition of $v \in \gfr$ for nilpotent $v_n$ and semisimple $v_s$ so that $v_s\in \tfr'$ for some Cartan subalgebra $\tfr'\subset \gfr$. Let $g\in G$ such that $g\cdot \tfr'=\tfr$. Then the adjoint quotient is given by
\begin{gather*}
\chi(v)=[g\cdot v_s]\in \tfr/W.
\end{gather*}
From this formula, it is clear that $\chi\colon \gfr\to \tfr/W$ is a generalization of the characteristic polynomial, see Exercise~\ref{exr:invariants}.

The smooth locus of the adjoint quotient $\chi\colon \gfr\to \tfr/W$ is given by
\begin{gather*}
\gfr^{\rm reg}=\{v\in \gfr\,|\,\dim \ker \operatorname{ad}(v)=r \} \subset \gfr,
\end{gather*}
so that $d\chi_v$ is surjective iff $v\in \gfr^{\rm reg}$. This result goes back to Kostant~\cite{KostantReps}.

\subsection[$G$-Hitchin systems]{$\boldsymbol{G}$-Hitchin systems}
A $G$-Higgs bundle over $C$ is a pair $(P,\varphi)$ consisting of a $G$-principal bundle $P\to C$ and a~section $\varphi\in H^0(C,K_C\otimes \operatorname{ad}(P))$, called a Higgs field, where $\operatorname{ad}(P)=P\times_G \gfr$ is the adjoint bundle. Any representation $\rho\colon G\to {\rm GL}(n,\C)$ induces a Higgs vector bundle $(\rho(P),\rho_*(\varphi))$. If $\rho=\operatorname{Ad}\colon G\to {\rm GL}(\mathfrak{g})$ is the adjoint representation, then we call $(\operatorname{Ad}(P),\operatorname{Ad}_*(\varphi))$ the adjoint Higgs (vector) bundle.
\begin{dfn}A $G$-Higgs bundle $(P,\varphi)$ is semistable iff the adjoint Higgs bundle $(\operatorname{Ad}(P),$ $\operatorname{Ad}(\varphi))$ is semistable, i.e.,
\begin{gather*}
\frac{\operatorname{deg} (F)}{\operatorname{rk}(F)} \leq \frac{\operatorname{deg}(\operatorname{Ad}(P))}{\operatorname{rk}(\operatorname{Ad}(P))}
\end{gather*}
for any proper subbundle $0 \subsetneq F \subsetneq \operatorname{Ad}(P)$ which is preserved by $\operatorname{Ad}_*(\varphi)$.
\end{dfn}
The set of isomorphism classes of semistable and topologically trivial $G$-Higgs bundles over~$C$ carries the structure of a complex analytic space~$\Higgs(C,G)$. The adjoint quotient $\chi\colon \gfr\to \tfr/W$ globalizes to the Hitchin map
\begin{gather*}
\Hit\colon \ \Higgs(C,G)\to \Bb(C,G):=H^0(C,(K_C\otimes \tfr)/W), \qquad [P,\varphi]\mapsto \chi\circ\varphi,
\end{gather*}
where the $\C^*$-action on $\tfr/W$ is induced by the natural $\C^*$-action on $\tfr$. The target $\Bb(C,G)$ of the Hitchin map is called the Hitchin base. Choosing generators $\chi_1,\dots, \chi_r\in \C[\gfr]^G$ of degree $d_i=\deg(\chi_i)$, we obtain the isomorphism
\begin{gather*}
\Bb(C,G)\cong \bigoplus_{i=1}^r H^0\big(C,K_C^{\otimes d_i}\big).
\end{gather*}
In particular, $\Bb(C,G)$ is a vector space in a non-canonical way.

\begin{thm}\label{thm:hitsys} Let $G$ be a semisimple complex Lie group and $C$ a compact Riemann surface of genus $\geq 2$. Then the smooth locus $\Hig(C,G)\subset \Higgs(C,G)$ of the moduli space of semistable and topologically trivial $G$-Higgs bundles is a holomorphic symplectic manifold. Moreover, the Hitchin map $\Hit\colon \mathbf{Higgs}(C,G)\to \Bb(C,G)$ is an algebraic integrable system. It admits sections, so-called Hitchin sections.
\end{thm}
Therefore, once a Hitchin section is chosen, the Hitchin fibers $\Hit^{-1}(b)$, for generic $b\in \Bb(C,G)$, are canonically isomorphic to abelian varieties $\mathsf{P}_b$. We give the isogeny class of~$\mathsf{P}_b$ in terms of cameral curves, which we treat in the next subsection.

\begin{rem}\label{rem:isogenyclass}\quad
\begin{enumerate}[label=\roman*)]\itemsep=0pt
\item It is possible to determine the isomorphism class of $\mathsf{P}_b$, for generic $b\in \Bb(C,G)$, in terms of cameral curves as well. However, this is much more subtle: The isomorphism class depends on the fundamental group $\pi_1(G)$ of $G$ whereas the isogeny class only depends on the Lie algebra of~$G$. We refer to~\cite{DG} and~\cite{DP} for more details.
\item The first two statements of Theorem~\ref{thm:hitsys} hold for any reductive complex Lie group $G$ as long as the topologically type of the $G$-Higgs bundles is fixed (see in particular~\cite{DP, Faltings}).
\end{enumerate}
\end{rem}

\subsection{Interlude: Cameral versus spectral curves}\label{ss:cameral}
Let $\iota\colon G\hookrightarrow {\rm GL}(n,\C)$ be a classical semisimple complex Lie group of rank $r$ and fix free generators $\chi_1,\dots, \chi_r$ of $\C[\gfr]^G$ (cf.\ Exercise~\ref{exr:invariants}). Then we identify $\Bb(C,G)=\oplus_{j=1}^r H^0\big(C,K_C^{\otimes d_j}\big)$ for $d_j=\deg(\chi_j)$. For a given $G$-Higgs bundle $(P,\varphi)$, we define the spectral curve $\tilde{C}_{\varphi,\iota}$ as
\begin{gather}\label{eq:specclassical}
\tilde{C}_{\varphi,\iota}=\{ \alpha\in \operatorname{tot}(K_C) \,|\,\det(\iota_*(\varphi) -\alpha\cdot \mathrm{id})=0\}.
\end{gather}
It parametrizes the (ordered) eigenvalues of $\iota_*(\varphi)$ and only depends on $b=(\chi_i\circ \phi)_{i=1,\dots, r}\in \Bb(C,G)$. Hence we simply write $\tilde{C}_{b,\iota}$. Hitchin determined the isomorphism classes of the fiber $\Hit^{-1}(b)$ for generic $b\in \Bb(C,G)$ in terms of the spectral curve $\tilde{C}_{b,\iota}$ \cite{Hit2, Hit1}.
\begin{ex}\label{ex:SL2fiber} Let $\iota\colon G={\rm SL}(2,\C)\hookrightarrow {\rm GL}(2,\C)$. Clearly, $\det\colon \gfr\to \C,$ $A\mapsto \det(A)$ is $G$-invariant of degree $2$ and generates $\C[\gfr]^G$. Hence the adjoint quotient $\chi\colon \gfr\to \tfr/W$ becomes $A\mapsto \det A$. The Hitchin map is then given by
\begin{gather*}
\Hit\colon \ \Hig(C,G)\to H^0\big(C,K_C^{\otimes 2}\big),\qquad [P,\varphi]\mapsto \det \varphi
\end{gather*}
where we identify $\Bb(C,G)= H^0\big(C,K_C^{\otimes 2}\big)$. For $b\in H^0\big(C,K_C^{\otimes 2}\big)$ the spectral curve is
\begin{gather*}
\tilde{C}_{b,\iota}=\big\{ \alpha\in \operatorname{tot}(K_C)\,|\,\alpha^2-b=0 \big\}.
\end{gather*}
If $b$ has simple zeros, then $\tilde{C}_{b,\iota}$ is smooth. In that case
\begin{gather*}
\Hit^{-1}(b)\cong \operatorname{Prym}\big(\tilde{C}_{b,\iota}/C\big)
\end{gather*}
for the Prym variety $\operatorname{Prym}\big(\tilde{C}_{b,\iota}/C\big)=\{ L\in \mathrm{Jac}\,|\,\sigma^*L=L^*\}$ of the double covering $\Ct_{b,\iota}\to C$ with covering involution $\sigma\colon \tilde{C}_{b,\iota}\to \tilde{C}_{b,\iota}$.
\end{ex}
Spectral curves are very concrete and convenient to work with but can be generically reducible (e.g., for $G={\rm SO}(2n+1,\C)$) or singular (e.g., for $G={\rm SO}(2n,\C)$). These are mild difficulties because one can work with appropriate irreducible components and normalizations respectively. A more serious drawback of spectral curves is that for a general semisimple complex Lie group~$G$ they depend on a representation $\rho\colon G\to {\rm GL}(n,\C)$. In particular, if $G$ is an exceptional Lie group, there is no canonical way to construct spectral curves. Donagi introduced cameral curves~\cite{Donagi2} to resolve these issues. Given $b\in \Bb(C,G)\cong \oplus_{j=1}^r H^0\big(C,K_C^{\otimes d_j}\big)$, the cameral curve $\tilde{C}_b\to C$ is defined as the pullback
\begin{equation}\label{eq:Cb}
\begin{tikzcd}
\tilde{C}_b \ar[r] \ar[d] & K_C\otimes \tfr \ar[d, "\bm{q}"] \\
C \ar[r, "b"] & (K_C\otimes \tfr)/W,
\end{tikzcd}
\end{equation}
which is a $W$-Galois covering for the Weyl group $W$\footnote{Occasionally, such cameral curves are called $K_C$-valued cameral curves. A more general concept are abstract cameral curves, see \cite{DG} for details.}

In this subsection we explain, based on \cite{Donagi2}, how cameral curves determine all spectral curves and how they resolve the aforementioned issues. For simplicity, we restrict to the local case, i.e., $C=\C$ is the complex line. The discussion readily generalizes to arbitrary Riemann surfaces $C$ if one twists with the canonical bundle~$K_C$.

Let $G$ be any semisimple complex Lie group and $\gfr$ be its Lie algebra. Any Higgs field $\varphi\colon C\to \gfr$ (for the trivial bundle $C\times P$) and representation $\rho\colon \gfr\to \mathfrak{gl}(V)$, $\dim_{\C}(V)=n$, defines a spectral curve
\begin{gather*}
\Ct_{\varphi,\rho}:=\{ (\zeta, z)\in C\times \C\,|\,\det(\rho\circ \varphi(\zeta)-z\cdot \mathrm{id})=0 \},
\end{gather*}
compare the classical case (\ref{eq:specclassical}). Again it parametrizes the eigenvalues of $\rho\circ\varphi$ over $C$ and decomposes as follows. The decomposition $\rho=\oplus_i \rho_i$ into irreducible representations $\rho_i$ clearly induces a corresponding decomposition of $\Ct_{\varphi,\rho}$.

However, even if $\rho\colon \gfr\to \mathfrak{gl}(V)$ is irreducible, $\Ct_{\varphi,\rho}$ decomposes further: Let \mbox{$R^s(\gfr)=\{ \alpha_i\,|\,i\!\in\! I\}$} be fixed simple roots with respect to a fixed Cartan subalgebra $\tfr\subset \gfr$, and let
\begin{gather*}
\mathsf{C}:=\{ v\,|\,\langle v,\alpha_i\rangle \geq 0 \, \forall\, i\in I \}
\end{gather*}
be the closed Weyl chamber in the corresponding root space $(V(\gfr),\langle \bullet, \bullet \rangle)$. Then the weight decomposition of the representation $V$ reads as
\begin{gather*}
V=\bigoplus_{\lambda\in \mathsf{C}} \bigoplus_{\mu\in W\cdot \lambda} V_{\mu}.
\end{gather*}
Every $\lambda\in \mathsf{C}$ defines a $G$-invariant $P_{\lambda}\in \C[\gfr]^G[z]$ by the requirement
\begin{gather}\label{eq:Plambda}
P_{\lambda}(t,z)=\prod_{\mu\in W\cdot \lambda} (\mu(t)-z), \qquad \forall\, (t,z)\in \tfr\times \C
\end{gather}
cf.\ Chevalley's theorem. This defines the curve
\begin{gather*}
\Ct_{\varphi,\lambda}=V(\varphi^*P_{\lambda})\subset C\times \C,
\end{gather*}
the vanishing locus of $\varphi^*P_{\lambda}\colon C\times \C \to \C$.
\begin{lem}\label{lem:decomp} Let $\rho\colon \gfr\to \mathfrak{gl}(V)$ be an irreducible representation of the semisimple complex Lie algebra $\gfr$, $\dim_{\C}(V)=n$ and $\varphi\colon C\to \gfr$ a Higgs field. Then the spectral curve $\Ct_{\varphi,\rho}$ decomposes into irreducible components
\begin{gather*}
\Ct_{\varphi,\rho}=\coprod_{\lambda\in \mathsf{C}} m_\lambda \Ct_{\varphi,\lambda}, \qquad m_{\lambda}=\dim V_{\lambda}.
\end{gather*}
\end{lem}
\begin{proof}This follows from the fact that
\begin{gather*}
\det(\rho\circ \varphi(\zeta)-z\cdot \mathrm{id})=\prod_{\lambda\in \mathsf{C}} P_{\lambda}(\varphi(\zeta)-z\cdot \mathrm{id})^{m_\lambda}.\tag*{\qed}
\end{gather*}\renewcommand{\qed}{}
\end{proof}

By construction, $\Ct_{\varphi,\rho}$ and $\Ct_{\varphi,\lambda}$ only depend on $b=\chi\circ \varphi$, so we denote $\Ct_{b,\rho}=\Ct_{\varphi,\rho}$ and $\Ct_{b,\lambda}=\Ct_{\varphi,\lambda}$ from now on. It is easy to see that the cameral curve $p_b\colon \Ct_b\to C$ factorizes as
\begin{equation*}
\begin{tikzcd}
\Ct_b \ar[dd] \ar[rd] & \\
& \Ct_{b,\lambda}. \ar[ld] \\
C
\end{tikzcd}
\end{equation*}
The curve $\Ct_{b,\lambda}$ is in general singular even if~$\Ct_b$ is smooth. In the following we assume that~$\Ct_b$ is smooth and construct finitely many smooth birational models of the infinitely many curves~$\Ct_{b,\lambda}$. To do so, let $W_\lambda\subset W$ be the stabilizer group of~$\lambda$. It is generated by a subset $J\subset I$ of simple roots. We denote by
\begin{gather*}
W_{J}=\langle s_{\alpha_j}\,|\,j\in J\rangle, \quad J\subset I
\end{gather*}
the generated subgroup of $W$. Clearly, there are only finitely many such subgroups. Moreover, $W_{\lambda}=W_{\lambda'}=W_J$ iff $\lambda,\lambda'$ lie in the same face $\mathsf{C}_J=\{ v\,|\,\langle v, \alpha_j\rangle=0,\, j\in J \} $ of the closed Weyl chamber $\mathsf{C}$.
\begin{rem}The subgroups $W_J\subset W$ for $J\subset I$ are called parabolic Weyl subgroups, since they correspond to parabolic subgroups $P\subset G$, see \cite[Section~30]{Humphreys}.
\end{rem}
For every parabolic subgroup $W_J\subset W$, the quotient $\Ct_{b,J}:=\Ct_b/W_J$ is smooth as well and defines the intermediate covering
\begin{gather*}
\Ct_{b,J}\to C.
\end{gather*}
If $\lambda\in \mathsf{C}_J$, we obtain the morphism
\begin{gather}\label{eq:birational}
p_{b,\lambda}\colon \ \Ct_{b,J}\to \Ct_{b,\lambda}, \qquad (z,[t]_P)\mapsto (z, \lambda(t)).
\end{gather}
Since $\Ct_{b,J}$ is smooth, $\Ct_{b,\lambda}$ is smooth if $p_{b,\lambda}$ is an isomorphism.
\begin{lem}[\cite{Donagi2}]\label{lem:specdecomp} Let $\lambda\in \wt_{\gfr}$ be a weight and let $J\subset I$ be such that $\lambda\in \mathsf{C}_J^\circ$, the interior of the face $\mathsf{C}_J\subset \mathsf{C}$.
Further, let $b\colon C\to \tfr/W$ be transversal to the discriminant of $q\colon \tfr\to \tfr/W$. Then $p_{b,\lambda}\colon \Ct_{b,J}\to \Ct_{b,\lambda}$ is birational. It is an isomorphism if $\lambda-w\cdot \lambda$ is a multiple of a root for all $w\in W$. In that case $\lambda$ is a multiple of a fundamental weight.
\end{lem}
\begin{proof}Let $P_{\lambda}\colon \tfr/W\to \C[z]$ be as defined in (\ref{eq:Plambda}) and consider
\begin{gather*}
(\tfr/W)_{\lambda}:=V(P_{\lambda})=\{ ([t],z)\in \tfr/W\times \C\,|\,P_{\lambda}([t],z)=0 \}.
\end{gather*}
The morphism
\begin{gather*}
i_{\lambda}\colon \ \tfr/W_J\to (\tfr/W)_{\lambda}, \qquad [t]_J\mapsto (\lambda(t),[t]),
\end{gather*}
is defined over $\tfr/W$ (for the obvious maps $\tfr/W_J\to \tfr/W \leftarrow (\tfr/W)_{\lambda}$). It fails to be an isomorphism if there exist $[t]_J\neq [t']_J\in \tfr^{\rm reg}/W_J$ with $[t]=[t']\in \tfr^{\rm reg}/W$ and $\lambda(t)=\lambda(t')$. In particular, $i^\circ_\lambda\colon \tfr^{\rm reg}/W_J\to (\tfr^{\rm reg}/W)_{\lambda}$ is a birational morphism because it is an isomorphism on the complement of the divisor defined by
\begin{gather}\label{eq:failureiso}
\prod_{w\in W-W_J} (\lambda-w\cdot \lambda)\colon \ \tfr^{\rm reg}/W_J\to \C.
\end{gather}
Since the morphism $p_{b,\lambda}\colon \tilde{C}_{b,J}\to \tilde{C}_{b,\lambda}$ is the pullback of $\iota_\lambda$ via $b$, a similar argument shows that $p_{b,\lambda}$ is a birational morphism as well.

A sufficient condition that $p_{b,\lambda}$ is an isomorphism is the non-vanishing of~(\ref{eq:failureiso}). Equivalently, $\lambda-w\cdot \lambda$ only vanishes along $\tfr-\tfr^{\rm reg}$ for all $w\in W\setminus W_J$. But then $\lambda-w\cdot \lambda$ must be a multiple of a root for all $w\in W$. This implies that $\lambda=m\omega$ for a fundamental weight $\omega$ and $m\in \Z$, see in \cite[Lemma~4.2]{Donagi2}.
\end{proof}

\begin{ex}Let $\gfr=\mathfrak{sl}(n+1,\C)$ be the simple complex Lie algebra of type $\mathrm{A}_{n}$. The associated root space is $\R^{n+1}/\Big\langle \sum\limits_{i=1}^{n+1} e_i \Big\rangle_{\R}$ with inner product induced by the standard inner product on~$\R^{n+1}$. We choose the simple roots $\alpha_i=\bar{e}_i-\bar{e}_{i+1}$ (where $\bar{e}_i$ is the class of $e_i$) with corresponding fundamental weights $\omega_j=\sum\limits_{i=1}^j \bar{e}_i$ and closed Weyl chamber~$\mathsf{C}$. The standard representation $\rho\colon \mathfrak{sl}(n+1,\C)\hookrightarrow \mathfrak{gl}(V)$, $V=\C^{n+1}$, has weights $\omega_1,\dots, \omega_n$ and $\omega_1$ is the only weight in the closed Weyl chamber~$\mathsf{C}$.

In particular, $V=\bigoplus_{\mu\in W\cdot \omega_1} V_\mu$, i.e., the representation is minuscule and each weight space is one-dimensional. The corresponding parabolic subgroup $W_J\subset W=S_{n+1}$ is generated by $s_{\alpha_2},\dots, s_{\alpha_n}\in W$ and is hence isomorphic to~$S_n$. In particular, if $\varphi\colon C\to \mathfrak{sl}(n+1,\C)$ is generic, then we recover the familiar spectral curve:
\begin{gather*}
\tilde{C}_b/W_J \cong \tilde{C}_{\varphi,\omega_1}\cong \tilde{C}_{\varphi, \rho}, \qquad b=\chi\circ \varphi.
\end{gather*}
\end{ex}
The previous example is misleading: Among the Dynkin types $\mathrm{A}_n$, $\mathrm{B}_n$, $\mathrm{C}_n$, $\mathrm{D}_n$, $\mathrm{G}_2$ the morphism $p_{b,\lambda}\colon \tilde{C}_{b,J}\to \tilde{C}_{b,\lambda}$ is an isomorphism only for the following multiples $\lambda$ of fundamental weights (see~\cite{Donagi2} and Exercise~\ref{exr:table}):
\begin{table}[h]
\centering
\begin{tabular}{c | c}
& $\lambda$
\\
\hline
$\mathrm{A}_n$ & $\omega_1$, $\omega_n$
\\
$\mathrm{B}_n$ ($n\geq 3$) & $\omega_1$, $\omega_2$
\\
$\mathrm{C}_n$ ($n\geq 3$) & $\omega_1$, $\omega_2$
\\
$\mathrm{D}_n$ & $-$
\\
$\mathrm{G}_2 $ & $\omega_1$, $\omega_2$
\end{tabular}
\caption{Weights $\lambda$ such that $p_{b,\lambda}$ is an isomorphism. \label{table}}
\end{table}
Lemmas \ref{lem:decomp} and \ref{lem:specdecomp} show how singularities of the spectral curves $\Ct_{b,\rho}$ arise even if $\tilde{C}_{b}$ is smooth and irreducible:
\begin{enumerate}[label=\roman*)]\itemsep=0pt
\item If the representation $\rho\colon \gfr\to \mathfrak{gl}(n,\C)$ is not minuscule, then $\Ct_{b,\rho}$ is reducible. Moreover, it is non-reduced if any of the weights $\lambda$ of $\rho$ has multiplicity $m_{\lambda}>1$.
\item Even if $\rho\colon \gfr\to \mathfrak{gl}(n,\C)$ is minuscule and all its weights have $m_{\lambda}=1$, $\Ct_{b,\rho}=\Ct_{b,\lambda}$ might be singular if $\lambda-w\cdot \lambda$ is \emph{not} a multiple of a root for all $w\in W$.

If $C$ is compact and we twist $\tfr/W_P$ and $(\tfr/W)_{\lambda}$ with the canonical bundle $K_C$, then $\Ct_{b,\lambda}$ is necessarily singular for such weights. For example, this is the case for $\gfr=\mathfrak{so}(2n+1,\C)$ and the standard representation $\rho\colon \gfr\hookrightarrow \mathfrak{gl}(2n+1,\C)$ (see Table~\ref{table}) and one has to work with the normalization of $\tilde{C}_{b,\rho}$.
\end{enumerate}
Even though cameral curves are mostly better behaved than spectral curves, we point out that the latter are much more convenient for actual computations. For example, if $\gfr=\mathfrak{sl}(n,\C)$, then the degree of the covering $\Ct_b\to C$ is $|W|=(n+1)!$ whereas the degree of $\Ct_{b,\rho}\to C$ is $n+1$ for the standard representation $\rho\colon \mathfrak{sl}(n,\C)\hookrightarrow \mathfrak{gl}(n,\C)$.

\subsection{Isogeny class of generic Hitchin fibers}
After this short interlude, we give the isogeny class of generic Hitchin fibers in terms of the cameral curves $\tilde{C}_b$, $b\in \Bb$, cf.~(\ref{eq:Cb}). These are smooth if
\begin{gather*}
b\in \Bb^\circ:=\{ b\in \Bb \,|\,b\mbox{ is transversal to }\mathrm{disc}(\qf)\},
\end{gather*}
where $\mathrm{disc}(\qf)$ is the discriminant locus of $\qf\colon K_C\otimes \tfr\to (K_C\otimes \tfr)/W$. It is proven in~\cite{Sco1} that $\Bbo\subset \Bb$ is a Zariski-dense open subset by using Bertini's theorem.

Finally, we need the cocharacter lattice
\begin{gather*}
\bLambda_G:=\operatorname{Hom}(\C^*, T),
\end{gather*}
where $T\subset G$ is a fixed maximal torus.

\begin{thm}[\cite{Sco1}]\label{thm:isogenoushit} Let $G$ be a semisimple complex Lie group and $C$ a compact Riemann surface of genus $g\geq 2$. Then the abelian variety $\mathsf{P}_b\cong \Hit^{-1}(b)$ is isogenous to $(\mathrm{Jac}(\tilde{C}_b)\otimes_{\Z} \bLambda_G)^W$ where the Weyl group acts diagonally from the left.
\end{thm}

\begin{ex}We compare Theorem \ref{thm:isogenoushit} with Example~\ref{ex:SL2fiber} in case $G={\rm SL}(2,\C)$. Let $T\subset G$ be the maximal torus of diagonal matrices so that $\bLambda_G\cong \Z$. The Weyl group $W=\Z/2\Z$ acts on $\mathrm{Jac}(\tilde{C}_b)$ via pullback and on $\bLambda_G$ in the natural way.
It follows that
\begin{gather*}
\mathrm{Jac}\big(\tilde{C}_b\otimes_{\Z} \bLambda_G\big)^W\cong \operatorname{Prym}\big(\tilde{C}_b/C\big),\qquad b\in \Bb^\circ.
\end{gather*}
Hence Theorem \ref{thm:isogenoushit} gives the isomorphism class of generic Hitchin fibers in this case. However, this is false in general.
\end{ex}

\subsection{Abstract Seiberg--Witten differential}\label{ss:hitabsw}
We next determine the holomorphic symplectic structure of $\Hig(C,G)$ over $\Bbo$ in terms of an abstract Seiberg--Witten differential~(\ref{ss:absw}). In order to do so, we first give a polarizable $\Z$-VHS $\VH^H$ of weight $1$ over $\Bbo$ which is isogenous to the polarizable $\Z$-VHS $\VH(\Hit)$ defined by $\Hit$. The $\Z$-VHS $\VH^H$ is defined via the universal cameral curve
\begin{gather*}
\bm{p}\colon \ \tC:=\operatorname{ev}^* (K_C\otimes \tfr)\to C\times \Bb \to \Bb
\end{gather*}
for the evaluation map $\operatorname{ev} \colon C\times \Bb\to (K_C\otimes \tfr)/W$. It is clear that $\tC_b=\tilde{C}_b$ as defined in~(\ref{eq:Cb}). If $\bLambda_G$ is the cocharacter lattice of the semisimple complex Lie group $G$, then we define the polarizable $\Z$-VHS
\begin{gather*}
\VH^H=\big( (\VH_{\Z}(\bm{p}^\circ)\otimes \bLambda_G)^W, F^\bullet (\VH_{\Oo}(\bm{p}^\circ)\otimes \tfr)^W \big)
\end{gather*}
of weight $1$ where $W$ acts diagonally from the left. It is shown in \cite{Beck} that $\VH^H$ is isogenous to\footnote{Formally, there is an issue with the weights of the VHS which we neglect in these notes. We refer the reader to \cite{Beck} for more details.}~$\VH(\Hit)^*$. By Theorem~\ref{thm:abSW}, any abstract Seiberg--Witten differential of~$\VH^H$ induces one on~$\VH(\Hit)$. A natural candidate is the section
\begin{gather*}
\bm{\lambda}_{\rm SW}\colon \ \Bbo\to F^1\VH^H_{\Oo}, \qquad \bm{\lambda}_{\rm SW}(b)=\lambda_b\in H^0\big(\tilde{C}_b, K_{\tilde{C}_b}\otimes \tfr\big)^W,
\end{gather*}
where $\lambda_b$ is the restriction of the tautological section on the total space of $K_C\otimes \tfr$ to $\tilde{C}_b$.

\begin{thm}[{\cite[Corollary 2]{Beck}}] Let $C$ be a compact Riemann surface of genus $g\geq 2$ and~$G$ a semisimple complex Lie group. Then $\bm{\lambda}_{\rm SW}\in H^0\big(\Bbo,F^1\VH^{H}_{\Oo}\big)$ is an abstract Seiberg--Witten differential and
\begin{gather*}
(\mathcal{J}(\VH_H), \omega_{\bm{\lambda}_{\rm SW}})\cong (\Hig^\circ(C,G),\omega_H)
\end{gather*}
as smooth algebraic integrable systems over $\Bbo$.
\end{thm}
\begin{cor} The holomorphic symplectic form $\omega_H$ is exact on $\Hig^\circ(C,G)$.
\end{cor}
\begin{proof}This follows directly from Corollary \ref{cor:exactness} and the previous theorem.
\end{proof}

This fact was known before (in fact on all of $\Hig(C,G)$) from the construction of \linebreak $\Hig(C,G)$. In contrast, our proof simply follows from the properties of the algebraic integrable system.

\subsection{Exercises}

\begin{exr}\label{exr:invariants}Let $G\subset {\rm GL}(n,\C)$ be a classical semisimple complex Lie group and $\gfr=\operatorname{Lie}(G)$.
\begin{enumerate}[label=\alph*)]\itemsep=0pt
\item Find explicit generators of $\C[\gfr]^G$. (Hint: If $\gfr=\mathfrak{so}(2m,\C)$, then $\det\in \C[\gfr]^G$ is the square root of the Pfaffian $\mathrm{pf}\in \C[\gfr]^G$.)
\item Specialize to $G={\rm SL}(n,\C)$. Determine the fibers of $\chi\colon \gfr\to \tfr/W$ where $\tfr\subset \gfr$ is the Cartan subalgebra of diagonal matrices in $\gfr$.
\item Check explicitly for $n=2$ that $d\chi_A\colon \mathfrak{sl}(2,\C)\to \tfr/W\cong\C$ is surjective iff $A\in \mathfrak{sl}(2,\C)$ is regular.
\end{enumerate}
\end{exr}

\begin{exr}\label{exr:table}Check Table \ref{table}.
\end{exr}

\section{Relation between Calabi--Yau integrable and Hitchin systems}\label{s:relation}
In the last two sections we have seen two complex integrable systems, namely compact Calabi--Yau integrable systems $\mathcal{J}^2(\Xt)\to \tilde{B}$ and Hitchin systems $\Hig(C,G)\to \Bb$.
A fundamental difference between the two is that the former is of index $1$ whereas the latter is of index $0$.
Hence they cannot be isomorphic to each other.

In this section we outline the construction of an algebraic integrable system out of non-compact Calabi--Yau threefolds, called non-compact Calabi--Yau integrable systems, and show that they are isomorphic to $G$-Hitchin systems. The first instances of such isomorphisms goes back to \cite{DDD, DDP} where $G$ is the simple adjoint complex Lie group with $\mathrm{ADE}$-Dynkin diagram. To `geometrically engineer' the simple complex group $G$ in general, we need (an extended version of) the McKay correspondence and Slodowy slices that are summarized in the following two subsections.

\subsection{McKay correspondence}
One part of the McKay correspondence \cite{McKay} is a bijection\footnote{In fact, this part was already known to du Val~\cite{duVal}; the full McKay correspondence~\cite{McKay} takes into account the irreducible representations of the groups as well.} between finite subgroups $\Gamma\subset {\rm SL}(2,\C)$ and irreducible $\ADE$-Dynkin diagrams. We next explain how this correspondence extends to arbitrary irreducible Dynkin diagrams following Slodowy~\cite{Slo}.

Let $\Delta$ be any irreducible Dynkin diagram. It is obtained by an irreducible $\ADE$-Dynkin diagram $\Delta_h$ via folding by graph automorphisms. More precisely, we need a special class of graph automorphisms.

\begin{dfn}Let $\Delta$ be an irreducible Dynkin diagram. A Dynkin graph automorphism is a graph automorphism $\tau\in \Aut(\Delta)$ such that $\alpha$ and $\tau(\alpha)$ are not neighbors for every vertex $\alpha\in \Delta$.
We denote by $\Aut_D(\Delta)\subset \Aut(\Delta)$ the subgroup of Dynkin graph automorphisms.
\end{dfn}

This condition is best understood on the level of root systems. Let $(R, (V,\langle \bullet, \bullet \rangle))$ be the root system $R=R(\Delta)$ corresponding to $\Delta$ and denote by $Q=\langle R \rangle_{\Z}\subset V$ the abelian subgroup generated by $R$. If $\tau \in \Autd(\Delta)$, then we define $\aD\in {\rm GL}(V)$ by
\begin{gather*}
\tau(e_{\alpha})=e_{\tau(\alpha)},
\end{gather*}
where $e_\alpha\in V$ is the basis vector corresponding to $\alpha\in R$. Hence $\tau$ is a Dynkin graph automorphism iff
\begin{gather*}
\langle \tau(\alpha), \alpha \rangle =0.
\end{gather*}
It is not difficult to see that
\begin{gather*}
\Aut_D(\Delta)=
\begin{cases}
1, & \Delta=\mathrm{A}_{2n},\\
\Aut(\Delta), & \Delta\neq \mathrm{A}_{2n}.
\end{cases}
\end{gather*}
In particular, for every subgroup $\Cc\subset \Autd(\Delta)$, there is an element $\tau \in \Cc$ of maximal order. Let $Q^\tau\subset Q=\langle R \rangle_{\Z}$ be the invariants and define
\begin{gather}\label{eq:aO}
R^{\Cc}:=R^\tau=\bigg\{ \alpha_O:=\sum_{\alpha'\in O(\alpha)} \alpha'~\Bigg|~\alpha\in R\bigg\}\subset Q^{\aD},
\end{gather}
where $O(\alpha)$ denotes the orbit of $\alpha\in R$ under $\aD$.
\begin{rem}
Even though an element $\tau\in \Cc$ of maximal order is not unique, $R^{\Cc}$ is well-defined.
\end{rem}
For the next lemma we introduce the group
\begin{gather*}
{\rm AS}(\Delta)=
\begin{cases}
\Z/2\Z,& \Delta=\mathrm{B}_k, \,\mathrm{C}_k,\, \mathrm{F}_4,\\
S_3,& \Delta=\mathrm{G}_2, \\
1,& \mbox{else}
\end{cases}
\end{gather*}
of symmetries associated with $\Delta$.

\begin{lem}\label{lem:folding}The mapping $(R,\Cc) \mapsto R^{\Cc}$ on the level of root systems induces the bijection\footnote{We borrow Slodowy's notation here and use the subscript $h$ in $\Delta_h$ for `homogeneous' which he uses synonymous to `simply-laced'.}
\begin{gather*}
(\Delta_h,\Cc)\mapsto \Delta=\Delta_h^\Cc, \\
(\Delta_h,{\rm AS}(\Delta)) \mapsfrom \Delta,
\end{gather*}
where the non-trivial cases are summarized in the following table
\begin{gather}\label{table:folding}
\begin{array}{@{}c|c@{}}
(\Delta_h,\Cc) & \Delta=\Delta_h^{\Cc} \\ \hline
(\mathrm{A}_{2k+1},\Z/2\Z) & \mathrm{B}_{k+1} \\
(\mathrm{D}_{k+1}, \Z/2\Z) & \mathrm{C}_{k} \\
(\mathrm{E}_6, \Z/2\Z) & \mathrm{F}_4 \\
(\mathrm{D}_4, S_3) & \mathrm{G}_2.
\end{array}
\end{gather}
In particular, any irreducible Dynkin diagram $\Delta$ of type $\mathrm{B}_k$, $\mathrm{C}_k$, $\mathrm{F}_4$, $\mathrm{G}_2$ $(\mathrm{BCFG}$-Dynkin diagram for short$)$ is obtained by folding of $(\Delta_h, \Cc)$ for a unique $\ADE$-Dynkin diagram $\Delta_h$ and subgroup $\Cc\subset \Autd(\Delta_h)$.
\end{lem}

\begin{proof}We first consider $R^{\aD}\subset V^{\aD}$. It is clear that $R^{\aD}$ spans $V^{\aD}$. By definition of Dynkin graph automorphisms, we compute
\begin{gather*}
\langle \alpha_O,\alpha_O\rangle:= \sum_{\alpha'\in O(\alpha)} \langle \alpha',\alpha'\rangle\neq 0,
\end{gather*}
cf.~(\ref{eq:aO}), hence $0\notin R^{\aD}$ and further
\begin{gather*}
\langle \alpha_O,\alpha_O^\vee \rangle=2.
\end{gather*}
It remains to show that the reflections $s_{\alpha_O}\colon V^{\aD}\to V^{\aD}$ preserve $R^\aD$. Then we claim
\begin{gather}\label{SloFormulaFolding}
s_{\alpha_O}=\prod_{\alpha'\in O(\alpha)} s_{\alpha'} \qquad \forall\, \alpha\in R,
\end{gather}
see Exercise~\ref{exr:SloFormula} below. This formula implies $s_{\alpha_O}(R^\aD)=R^\aD$: Using $\aD s_\alpha \aD^{-1}=s_{\aD(\alpha)}$ we see from~(\ref{SloFormulaFolding}) that
\begin{gather*}
s_{\alpha_O}(\beta_O)=\bigg(\bigg(\prod_{\alpha'\in O(\alpha)} s_{\alpha'}\bigg)(\beta)\bigg)_O\in R^\aD,\qquad \forall\, \beta\in R.
\end{gather*}
This makes sense because $\aD$ acts cyclically. It is now straightforward (since simple roots of~$R$~($R^\vee$) give simple roots in~$R^{\aD}$ ($R^{\vee,\aD}$)) to compute the different types as claimed in the above table.
\end{proof}

\begin{rem}The name `folding' becomes evident if we depict the action of the graph automorphisms $\Cc$ on~$\Delta_h$ and `fold' $\Delta_h$ correspondingly. As an illustration we give the following example:
\begin{figure}[h]\centering
\begin{tikzpicture}
\coordinate (A) at (0,0) {};
\coordinate (B) at (1,0) {};
\coordinate (C) at (2,0) {};
\coordinate (D) at (3,0) {};
\coordinate (E) at (4,0) {};
\coordinate (F) at (5,0) {};
\coordinate (G) at (6,0) {};
\coordinate (H) at (7,0) {};
\coordinate (I) at (8,0) {};
\coordinate (J) at (9,0) {};
\coordinate (A5) at (-1.5,.25) {};
\coordinate (Cc) at (-1.5,-.25) {};
\coordinate (B4) at (10.5,0) {};
\coordinate (Bhalf) at (1.5, -0.9) {};
\coordinate (Hhalf) at (1.5,-1.45) {};

\node [fill=black, circle, inner sep=0pt, minimum size=5pt] at (A) {};
\node [fill=black, circle, inner sep=0pt, minimum size=5pt] at (B) {};
\node [fill=black, circle, inner sep=0pt, minimum size=5pt] at (C) {};
\node [fill=black, circle, inner sep=0pt, minimum size=5pt] at (D) {};
\node [fill=black, circle, inner sep=0pt, minimum size=5pt] at (E) {};
\node [fill=black, circle, inner sep=0pt, minimum size=5pt] at (H) {};
\node [fill=black, circle, inner sep=0pt, minimum size=5pt] at (I) {};
\node [fill=black, circle, inner sep=0pt, minimum size=5pt] at (J) {};
\node at (A5) {$\Delta_h=\mathrm{A}_5$};
\node at (B4) {$\Delta=\mathrm{B}_3$};
\node at (Cc) {$\Cc=\Z/2\Z$};

\draw (A) -- (B);
\draw (B) -- (C);
\draw (C) -- (D);
\draw (C) -- (E);
\draw [<->, shorten <= 0.15cm, shorten >=0.15cm] (A) to [in=90, out=90] (E);
\draw [<->, shorten <= 0.15cm, shorten >=0.15cm] (B) to [in=90, out=90] (D);
\draw (H) -- (I);
\draw[<->] (F) to (G);

\draw(8,-0.04) -- (9,-0.04);
\draw(8,0.04) -- (9,0.04);

\draw
(8.6,0) --++ (-135:.2)
(8.6,0) --++ (135:.2);

\end{tikzpicture}
\caption{Folding of $\Delta_h=\mathrm{A}_5$ to $\Delta_h^{\Cc}=\Delta=\mathrm{B}_3$.}\label{Figure1}
\end{figure}
\end{rem}

\begin{thm}[McKay correspondence] There is a one-to-one correspondence
\begin{gather}\label{eq:mckay}
\{ \Gamma\trianglelefteq \Gamma' \subset {\rm SL}(2,\C) \mbox{ finite subgroups} \} \leftrightarrow \{ \Delta \mbox{ irreducible Dynkin diagram}\} ,\\
\Gamma\trianglelefteq \Gamma' \mapsto \Delta=\Delta_h^{\Cc}, \qquad \Delta_h=\Delta_h(\Gamma),\qquad \Cc=\Gamma'/\Gamma. \nonumber
\end{gather}
\end{thm}
\begin{proof}[Sketch of proof] Let $\Gamma=\Gamma'$ so that $\Cc=1$. The quotient orbifold $\C^2/\Gamma$ has a unique singularity at $[0]\in \C^2/\Gamma$. It admits the minimal resolution
\begin{gather*}
\pi\colon \ \widehat{\C^2/\Gamma}\to \C^2/\Gamma.
\end{gather*}
Its resolution graph , i.e., the intersection graph of the irreducible components of its exceptional divisor, is dual to an irreducible $\ADE$-Dynkin diagram $\Delta_h(\Gamma)$. It turns out that the resolution graph uniquely determines $\Gamma$. Hence the mapping $\Gamma \mapsto \Delta_h(\Gamma)$ is a bijection onto irreducible $\ADE$-Dynkin diagrams.

Now let $\Gamma \trianglelefteq \Gamma'$ be two finite and distinct subgroups of ${\rm SL}(2,\C)$. Then the group $\Cc:=\Gamma'/\Gamma$ acts on $\C^2/\Gamma$. By the uniqueness of the minimal resolution, $\Cc$ also acts on $\widehat{\C^2/\Gamma}$, in particular on the dual $\Delta_h(\Gamma)$ of its resolution graph. In that way we obtain the full bijection~(\ref{eq:mckay}).
\end{proof}

\begin{dfn}Let $\Delta=\Delta_h^\Cc$ be an irreducible Dynkin diagram where $(\Delta_h,\Cc)$ is as in~(\ref{table:folding}). Further let $\Gamma\trianglelefteq \Gamma'$ correspond to $\Delta$ under the McKay correspondence~(\ref{eq:mckay}). A~$\Delta$-singularity is a~(germ of a) surface singularity~$(Y,0)$ together with the action of a finite subgroup $H\subset \Aut(Y,0)$ which is isomorphic to $(\C^2/\Gamma, [0])$ together with the action of $\Cc\cong \Gamma'/\Gamma$.
\end{dfn}

Any $\Delta$-singularity $(Y,H)$ admits a semi-universal deformation. This is a semi-universal deformation $\sigma\colon \mathcal{Y}\to (B,0)$ of $Y$ such that $H$ acts fiber-preserving on $\mathcal{Y}$ inducing the $H$-action on $\sigma^{-1}(0)\cong Y$.

\begin{ex}\label{ex:B2} Let $\Delta=\mathrm{B}_2=\Delta_h^{\Cc}$ for $(\Delta_h,\Cc)=(\mathrm{A}_3,\Z/2\Z)$. The corresponding pair $\Gamma\trianglelefteq \Gamma'$ is given by
\begin{gather*}
\Gamma= \left\{ \begin{pmatrix}
\exp\big(\tfrac{\pi {\rm i}}{2}k\big) & 0 \\
0 & \exp\big({-}\tfrac{\pi {\rm i}}{2}k\big)
 \end{pmatrix}, \ k=0,1,2,3
\right\}
\trianglelefteq
\Gamma'= \left\langle \Gamma, \begin{pmatrix}
0 & 1 \\ -1 & 0
\end{pmatrix}\right\rangle.
\end{gather*}
Choosing appropriate generators of $\C[x,y]^\Gamma$, the quotient $\C^2/\Gamma$ is expressed as the hypersurface singularity $u^4-vw=0$ in $\C^3$. The action by $\Cc\cong \Gamma'/\Gamma=\Z/2\Z$ is $-1\cdot(u,v,w)=(-u, w,v)$. A~semi-universal deformation of this $\Delta$-singularity is given by
\begin{gather*}
\mathcal{Y}=\big\{ (u,v,w,(a_1,a_2))\in \C^3\times \C^2\,|\,u^4-vw+a_1 x^2+a_2=0 \big\} \to \C^2
\end{gather*}
with $\Cc$-action $(u,v,w,(a_1,a_2))\mapsto (-u,w,v,(a_1,a_2))$.
\end{ex}

\subsection{Slodowy slices}
Even though the quotient singularities $\C^2/\Gamma$ with the group action of $\Cc\cong \Gamma'/\Gamma$ are classified by irreducible Dynkin diagrams $\Delta=\Delta_h(\Gamma)^{\Cc}$, $\Cc={\rm AS}(\Delta)$, it is a priori unclear how they are related to the simple complex Lie algebra $\gfr=\gfr(\Delta)$ corresponding to $\Delta$. This relation was elucidated by Brieskorn~\cite{Bri} and Slodowy~\cite{Slo} after a suggestion by Grothendieck.

Let $x\in \gfr$ be a nilpotent element which is subregular, i.e.,
\begin{gather*}
\dim \ker \operatorname{ad}(x)=r+2, \qquad r=\operatorname{rk}(\gfr).
\end{gather*}
By the Jacobson--Morozov theorem \cite{CollingwoodMcGovern}, there exists a monomorphism $\rho\colon \mathfrak{sl}(2,\C)\hookrightarrow \gfr$ and $y,h\in \gfr$ such that
\begin{gather*}
\rho\begin{pmatrix} 0 & 1 \\ 0 & 0 \end{pmatrix}=x, \qquad \rho\begin{pmatrix} 0 & 0 \\ 1 & 0 \end{pmatrix}=y,\qquad \rho\begin{pmatrix} 1 & 0 \\ 0 & -1 \end{pmatrix}=h.
\end{gather*}
The triple $(x,y,h)$ is called an $\mathfrak{sl}_2$-triple for $x$. The Slodowy slice through $x$ is defined by
\begin{gather*}
S:=x+\ker \operatorname{ad}(y).
\end{gather*}
It carries a $\C^*\times\Cc$-action as we next explain. For the $\C^*$-action observe that $\rho\colon \mathfrak{sl}(2,\C)\hookrightarrow \gfr$ exponentiates to yield the group homomorphism $\tilde{\rho}\colon {\rm SL}(2,\C)\to G=G_{\rm ad}$. Letting $G$ act on $\gfr$ by the adjoint action, we define
\begin{gather*}
\C^*\times S\to S,\qquad (\zeta,v)\mapsto \zeta^2\tilde{\rho}\begin{pmatrix}
\zeta^{-1} & 0 \\
0 & -\zeta^{-1}
\end{pmatrix}\cdot v.
\end{gather*}
This is a well-defined action, i.e., it preserves $S$. For the $\Cc$-action define the group
\begin{gather*}
C(x,h):=\{ g\in G_{\rm ad} \,|\,g\cdot x=x, \, g\cdot h=h\},
\end{gather*}
which acts on $S\subset \gfr$ via the adjoint action. Let $C(x,h)^\circ\subset C(x,h)$ be the connected component of the identity. Then $C(x,h)/C(x,h)^\circ\cong \Cc$ and the exact sequence
\begin{equation*}
\begin{tikzcd}
1 \ar[r] & C(x,h)^\circ \ar[r] & C(x,h) \ar[r] & \Cc \ar[r] & 1
\end{tikzcd}
\end{equation*}
splits \cite[Section~7.5]{Slo}. In that way, $\Cc$ acts on $S$ as well. Since the $\C^*$-action commutes with the $\Cc$-action, we obtain a $\C^*\times \Cc$-action on $S$. It interacts with the restriction
\begin{gather*}
\sigma=\chi_{|S}\colon \ S \to \tfr/W
\end{gather*}
of the adjoint quotient $\chi\colon \gfr\to \tfr/W$ as follows. If we let $\C^*$ act on $\tfr/W$ by
\begin{gather*}
\zeta\cdot [t]=\big[\zeta^2 t\big],
\end{gather*}
i.e., by twice the standard weights, and $\Cc$ act trivially, then $\sigma\colon S\to \tfr/W$ is $\C^*\times \Cc$-equivariant. The importance of Slodowy slices for our purposes is summarized in the following.
\begin{thm}[Slodowy] Let $\Delta=\Delta_h^{\Cc}$ be an irreducible Dynkin diagram and $\gfr=\gfr(\Delta)$ the corresponding simple complex Lie algebra. If $x\in \gfr(\Delta)$ is a subregular nilpotent element and $S=x+\ker \operatorname{ad}(y)$ a Slodowy slice through $x$, then the following holds:
\begin{enumerate}[label=$\alph*)$]\itemsep=0pt
\item $(S_{[0]}, x)$ together with the induced $\Cc$-action is a $\Delta$-singularity.
\item The restriction $\sigma=\chi_{|S}\colon S\to \tfr/W$ of the adjoint quotient together with the $\Cc$-action is a~semi-universal deformation of the $\Delta$-singularity $(S_{[0]},x)$.
\end{enumerate}
\end{thm}

\subsection{Non-compact Calabi--Yau integrable systems}
As before we fix an irreducible Dynkin diagram $\Delta=\Delta_h^{\Cc}$ and $S\subset \gfr=\gfr(\Delta)$ a Slodowy slice. Moreover, we fix a compact Riemann surface $C$ of genus $\geq 2$ together with a spin bundle $L$, i.e., $L^2=K_C$. This data defines a family $\mathcal{X}\to \Bb=H^0(C, (K_C\otimes \tfr)/W)$ of threefolds via the diagram\footnote{Here we abuse notation and identify a line bundle with the $\C^*$-bundle obtained by removing the zero section.}
\begin{equation}\label{eq:constr3folds}
\begin{tikzcd}
\mathcal{X} \ar[r] \ar[d] \ar[dd, bend right, in=270, out=270,"\bm{\pi}"''] & L\times_{\C^*} S \ar[d, "\bm{\sigma}"] \\
C\times \Bb \ar[r, "{\rm ev}"] \ar[d] & L\times_{\C^*}\tfr/W \\
\Bb,
\end{tikzcd}
\end{equation}
where the square is the fiber product. Here we let $\C^*$ act by twice its weight so that $L\times_{\C^*} \tfr/W\cong (K_C\otimes \tfr)/W$. Since the $\C^*$-action on~$S$ commutes with the $\Cc$-action, $\pib\colon \X\to \Bb$ carries a fiber-preserving $\Cc$-action.

\begin{prop}[\cite{Beck2}]\label{prop:cy3s} Let $L\to C$ be a spin bundle and let $\pib\colon \X\to \Bb$ be the family of threefolds with $\Cc$-action constructed in~\eqref{eq:constr3folds}.
Then the following hold:
\begin{enumerate}[label=$\roman*)$]\itemsep=0pt
\item Every $X_b$, $b\in \Bb$, is quasi-projective and Gorenstein so that the canonical sheaf $\omega_{X_b}$ is a~line bundle $K_{X_b}$.
\item For every $b\in \Bb$, there is a nowhere-vanishing section $s_b\in H^0(X_b,K_{X_b})$ which is $\Cc$-invariant. In particular, every $X_b$ is a quasi-projective $($possibly singular$)$ Calabi--Yau threefold.
\item If $b\in \Bbo$, then $X_b$ is smooth.
\end{enumerate}
\end{prop}
A closely related construction, without the $\Cc$-action, appeared in \cite{DDP} inspired by~\cite{Sz1}. Instead of commenting on the proof of this proposition, we give an example.
\begin{ex}Let $\Delta=\mathrm{B}_2$ so tha $\Bb=H^0\big(\cu,K_C^{\otimes 2}\big) \oplus H^0\big(\cu, K_C^{\otimes 4}\big)$. Using the explicit description in Example~\ref{ex:B2}, the corresponding family $\X\to \Bb$ is written as
\begin{gather*}
\mathcal{X}=\big\{ (\alpha,\beta,\gamma,(b_1,b_2))\,|\,\alpha^{\otimes 4}-\beta\otimes\gamma+b_1\otimes \alpha^{\otimes 2}+b_2=0 \in \operatorname{tot}\big(K_C^4\big) \big\}
\end{gather*}
in $\operatorname{tot}\big(K_C\oplus K_C^2 \oplus K_C^2\big)\times \Bb$ with $\Cc$-action $(-1)\cdot (\alpha,\beta,\gamma,(b_1,b_2))=(-\alpha,\gamma,\beta,(b_1,b_2))$.
\end{ex}
For every $b\in \Bbo$, we define the $\Cc$-invariant intermediate Jacobian
\begin{gather*}
J^2_{\Cc}(X_b)=H^3(X_b,\C)^{\Cc}/\big(F^2H^3(X_b,\C)^{\Cc}+H^3(X_b,\Z)^\Cc\big).
\end{gather*}
Observe that these do not coincide with the $\Cc$-fixed points in $J^2(X_b)$.
\begin{thm}[\cite{Beck}]\label{thm:ncIS} Let $\pi^\circ\colon \X^\circ\to \Bbo$ be the smooth family of quasi-projective Calabi--Yau threefolds obtained in Proposition~{\rm \ref{prop:cy3s}}. Then
\begin{gather*}
\mathcal{J}^2_{\Cc}(\X^\circ/\Bbo)\to \Bbo
\end{gather*}
carries the structure of an algebraic integrable system.
\end{thm}

The previous theorem gives an example of a non-compact Calabi--Yau integrable system (also see~\cite{DDP}). Here `non-compact' refers to the fact that the underlying Calabi--Yau threefolds are non-compact. For that reason we cannot apply the result of Donagi--Markman (Theorem~\ref{thm:DM}) because it requires deformation theory of compact Calabi--Yau threefolds. Another important difference is that no base change is required.
\begin{proof}[Sketch of proof] Since $X_b$, $b\in \Bbo$, is only quasi-projective, $H^3(X_b,\Z)$ a priori carries a mixed Hodge structure~\cite{DeligneII}. In the first step, we prove that it is in fact pure of weight~$1$ and index~$0$. Hence it induces a polarizable $\Z$-VHS $\VH^{\rm CY}$ of weight~$1$. In the second step, we construct an abstract Seiberg--Witten differential $\blambda_{\rm CY}\colon \Bbo\to \VH^{\rm CY}_{\Oo}$ which is $\Cc$-invariant.
This shows that $\mathcal{J}^2_{\Cc}(\X^\circ)\to \Bbo$ is an algebraic integrable system.
\end{proof}

\subsection[Isomorphism with $G$-Hitchin systems]{Isomorphism with $\boldsymbol{G}$-Hitchin systems}
After all this preparation, we finally come to the precise relation between Calabi--Yau integrable and Hitchin systems.
\begin{thm}[\cite{Beck}]\label{thm:nccyhit} Let $\Delta=\Delta_{h}^{\Cc}$ be an irreducible Dynkin diagram and $G=G_{\rm ad}$ the corresponding simple adjoint complex Lie group with Hitchin base $\Bb=\Bb(C,G_{\rm ad})$. Further let $\mathcal{J}^2_{\Cc}(\X^\circ/\Bbo)\to \Bbo$ be a non-compact Calabi--Yau integrable system constructed in Theorem~{\rm \ref{thm:ncIS}}. Then $\mathcal{J}^2_{\Cc}(\X^\circ/\Bbo)\cong \Hig^\circ(C,G)$ as algebraic integrable systems over $\Bbo$.
\end{thm}
If $\Delta=\Delta_h$ is an irreducible $\ADE$-Dynkin diagram, then the fiberwise version of Theorem~\ref{thm:nccyhit} appeared in~\cite{DDP} (see~\cite{DDD} for~$\mathrm{A}_1$).
\begin{proof}[Sketch of proof] Let $\VH^{\rm CY}$ and $\VH^H$ be the $\Z$-VHS of weight $1$ determined by $\pib^\circ\colon \X^\circ\to \Bbo$ and $\Hit^\circ\colon \Hig(C,G)^{\circ}\to \Bbo$ respectively. We denote by $\VH^{\rm CY}_{\Cc}\subset \VH^{\rm CY}$ the $\Cc$-invariant part. From our preparatory work, we know that both $\VH^{\rm CY}_{\Cc}$ and $\VH^H$ admit abstract Seiberg--Witten differentials $\blambda_{\rm CY}$ and $\blambda_{H}$ respectively. Using M.~Saito's mixed Hodge modules~\cite{SaitoMixedHodgeModules}, which are a~vast generalization of variations of Hodge structures, we prove an isomorphism $\Psi\colon \VH^{\rm CY}_{\Cc}\rightarrow \VH^H$
of polarizable $\Z$-VHS. Here are a subtle point is to match the integral structures. In yet another step, we prove that $\Psi$ intertwines $\blambda_{\rm CY}$ with $\blambda_{H}$. An application of Proposition \ref{thm:abSW} then implies that $\mathcal{J}^2_{\Cc}(\X^\circ/\Bbo)\to \Bbo$ and $\Hig^\circ(C,G)\to \Bbo$ are isomorphic as algebraic integrable systems over $\Bbo$.
\end{proof}

By working with compactly supported cohomology instead of ordinary cohomology, we obtain Theorem~\ref{thm:nccyhit} for the Langlands dual group~$^L G$ of~$G$, see~\cite{Beck}. In particular, this proves Theorem~\ref{thm:nccyhit} for all simply-connected simple complex Lie groups. Applying Poincar\'e duality (more precisely Poincar\'e--Verdier duality, see \cite[Chapter~3]{KashiwaraSchapira}, we conclude that
\begin{gather*}
\mathcal{J}_{\Cc}^2(\Xo/\Bbo)\cong \mathcal{J}_2^{\Cc}(\Xo/\Bbo)^\vee
\end{gather*}
over $\Bbo$. Combined with Theorem \ref{thm:nccyhit} this gives a proof of the Langlands duality statement
\begin{gather*}
\Hig^\circ(C,G)\cong \Hig^\circ\big(C,^L G\big)^\vee
\end{gather*}
over $\Bbo$ which is independent of \cite{DP}.

Finally, in \cite{Beck2} we prove a geometric version of Theorem~\ref{thm:nccyhit} by using Calabi--Yau orbifolds, i.e., global quotient stacks with trivial canonical class. In that way taking $\Cc$-invariants is already `built in'.

\subsection{Open questions}
Even though the basic relationship between Calabi--Yau integrable and Hitchin systems has been established, there are many open questions. We close this article with two of them.
\begin{itemize}\itemsep=0pt
\item (Geometric objects on CY3s) There is no immediate modular interpretation of the intermediate Jacobians $J^2(X_b)$, $b\in \Bbo$.
Hence a natural question is, are there geometric objects on $X_b$, $b\in \Bbo$, that explicitly correspond to Higgs bundles on~$C$.
\item (Branes) As mentioned in the introduction, the total space $\Hig(C,G)$ of a $G$-Hitchin system is a Hyperk\"ahler manifold. In particular, it carries three complex structures $I_k$ as well as three corresponding symplectic/K\"ahler forms $\omega_k$ ($k=1,2,3$).

If we fix $k$, then a submanifold $L\subset \Hig(C,G)$ is called an $A/B$-brane if $L$ is a Lagrangian/complex submanifold with respect to $I_k$\footnote{This terminology orginated in physics (see, e.g.,~\cite{KapustinWitten}). The actual `physical branes' carry additional data.}. A submanifold maybe Lagrangian or complex with respect to any of the three complex structures. Therefore it makes sense to speak of $(B,B,B)$-, $(B,A,A)$-, $(A,B,A)$- and $(A,A,B)$-branes. A very active area of research is to find such branes (see \cite[Section~5]{AndersonEtAlAIM} for an overview). Is it possible to construct known or new branes in $\Hig(C,G)$ using the Calabi--Yau threefolds $\X\to \Bb(C,G)$?
\end{itemize}

\subsection{Exercises}

\begin{exr}\label{exr:SloFormula}Show formula (\ref{SloFormulaFolding}) in the proof of Lemma \ref{lem:folding}.
\end{exr}

\begin{exr}Let $\gfr_n=\mathfrak{sl}(n,\C)$ with its standard Cartan subalgebra $\tfr_n\subset \gfr_n$ of diagonal matrices.
\begin{enumerate}[label=\alph*)]\itemsep=0pt
\item 
Give an example of a Slodowy slice $S_n\subset \gfr_n$.
\item Let $S_3\subset \gfr_3=\mathfrak{sl}(3,\C)$ be a Slodowy slice. Use Exercise~\ref{exr:invariants} to check directly that the fiber of $\sigma\colon S_3\to \tfr_3/W \cong \C^2$ over $0$ is isomorphic to the $\mathrm{A}_2$-singularity $u^3-vw=0$.
\end{enumerate}
\end{exr}

\begin{exr}Let $\gfr$ be a semisimple complex Lie group and $(x,y,h)$ a \emph{regular} nilpotent element. Prove that the `regular Slodowy slice' $S_{\rm reg}=x+\ker \operatorname{ad}(y)\subset \gfr$, which is more commonly known as Kostant slice, is a section of the adjoint quotient $\chi\colon \gfr\to \tfr/W$. In particular, $\chi_{|S_{\rm reg}}\colon S_{\rm reg}\to \tfr/W$ is an isomorphism.
\end{exr}

\subsection*{Acknowledgements}
It is a pleasure to thank Lara Anderson and Laura Schaposnik for the invitation to contribute to the SIGMA Special Issue on Geometry and Physics of Hitchin Systems and to give lectures on the topics of these notes at the `Workshop on the Geometry and Physics of Higgs bundles IV' on March 16--17, 2019. Moreover, I~thank Murad Alim, Aswin Balasubramanian, Peter Dalakov and Martin Vogrin for comments on the first draft of these notes. This work is financially supported by the NSF grant ``NSF CAREER Award DMS 1749013'', the Simons Center for Geometry and Physics and the DFG Emmy-Noether grant on ``Building blocks of physical theories from the geometry of quantization and BPS states'', number AL 1407/2-1.

\pdfbookmark[1]{References}{ref}
\LastPageEnding

\end{document}